\documentclass[10pt,a4paper]{article}

\usepackage{fancyhdr}
\usepackage[top=1in, bottom=1.25in, left=1.25in, right=1.25in]{geometry}
\usepackage{amsmath}
\usepackage{amsfonts}
\usepackage{amssymb}
\usepackage{makeidx}
\usepackage{amsthm}
\usepackage[all]{xy}
\usepackage{graphicx}
\usepackage{undertilde}
\usepackage{standalone}
\usepackage{graphicx}
\usepackage{tikz-cd}[2013/12/13]
\usepackage{MnSymbol}
\usepackage[compact,explicit]{titlesec}
\usepackage{marginnote}
\usepackage{stmaryrd}

\newcommand\cyr{%
\renewcommand\rmdefault{wncyr}%
\renewcommand\sfdefault{wncyss}%
\renewcommand\encodingdefault{OT2}%
\normalfont
\selectfont}
\DeclareTextFontCommand{\textcyr}{\cyr}

\newtheorem{thm}{Theorem}[section]
\newtheorem{crl}[thm]{Corollary}
\newtheorem{lmm}[thm]{Lemma}
\newtheorem{prp}[thm]{Proposition}
\newtheorem{xmp}[thm]{Example}
\newtheorem{rmk}[thm]{Remark}
\newtheorem{dfn}[thm]{Definition}
\newtheorem{ntt}[thm]{Notations}
\newtheorem{cnt}[thm]{Construction}
\newtheorem{innercustomthm}{Theorem}
\newenvironment{customthm}[1]
  {\renewcommand\theinnercustomthm{#1}\innercustomthm}
  {\endinnercustomthm}

\DeclareMathOperator{\orb}{\mathcal{O}}
\DeclareMathOperator{\eqdef}{\text{:=}}
\DeclareMathOperator{\Dec}{Dec}
\DeclareMathOperator{\basis}{\mathcal{S}}
\DeclareMathOperator{\tbasis}{\mathcal{T}}
\DeclareMathOperator{\gbasis}{\orb\basis}

\DeclareMathOperator{\Tree}{\tau}
\DeclareMathOperator{\sTree}{\mathtt{t}}
\DeclareMathOperator{\tTree}{\mathfrak{t}}
\DeclareMathOperator{\gTree}{\orb\mathfrak{t}}
\DeclareMathOperator{\shuffle}{Sh}
\newcommand\restr[2]{{\left.\kern-\nulldelimiterspace 
  \vphantom{\big|}  \right|_{#2}  }}
\DeclareMathOperator{\Mod}{Mod_{\mathbb{K}}}
\DeclareMathOperator{\ModSigma}{Mod_{\mathbb{K}}^{\SymGrp{}}}
\DeclareMathOperator{\ModN}{Mod_{\mathbb{K}}^{\mathbb{N}}}

\DeclareMathOperator{\Ind}{Ind}
\newcommand{\Germ}{\raise -10pt\hbox{ \begin{picture}(20,20)
\put(7,16){\circle*{3}}\put(1,16){$\scriptscriptstyle 2$}
\put(7,6){\line(0,1){10}}
\put(7,6){\circle{3}} \put(1,6){$\scriptscriptstyle 1$}
\end{picture}}}
\newcommand{\Root}{\hbox{ \begin{picture}(10,10)
\put(0,2){\circle{3}}\put(3,0){$\scriptscriptstyle 1$}
\end{picture}}}
\newcommand{\FlowerTwo}[3]{
\raise -10pt\hbox{ \begin{picture}(35,30) \put(10,6){\circle{3}}   \put(0,16){\circle*{3}}   \put(20,16){\circle*{3}} 
\put(10,6){\line(-1,1){10}}  \put(10,6){\line(1,1){10}}
\put(10,-1){$\scriptscriptstyle #1$}   \put(-4,20){$\scriptscriptstyle #2$}   \put(20,20){$\scriptscriptstyle #3 $}  
\end{picture}}
}
\newcommand{\Flower}[1]{
\raise -10pt\hbox{ \begin{picture}(35,30) \put(10,6){\circle{3}}   \put(0,16){\circle*{3}}   \put(20,16){\circle*{3}} \put(10,18){\circle*{1}}  \put(4,17){\circle*{1}} \put(7,18){\circle*{1}} \put(16,17){\circle*{1}} \put(13,18){\circle*{1}}
\put(10,6){\line(-1,1){10}}  \put(10,6){\line(1,1){10}}
\put(10,-1){$\scriptscriptstyle 1$}   \put(-4,20){$\scriptscriptstyle 2$}   \put(20,20){$\scriptscriptstyle #1 $}  
\end{picture}}
}
\newcommand{\LFlower}[4]{\raise -10pt\hbox{ \begin{picture}(35,35)
\put(4,26){\circle*{3}}\put(-2,26){$\scriptscriptstyle #4$}\put(4,16){\line(0,1){10}}
    
\put(4,16){\circle*{3}}\put(-2,16){$\scriptscriptstyle #2$} \put(24,16){\circle*{3}} \put(26,16){$\scriptscriptstyle #3$}  
\put(14,6){\line(-1,1){10}} \put(14,6){\line(1,1){10}}
    
\put(14,6){\circle{3}} \put(14,-1){$\scriptscriptstyle #1$}
\end{picture}}}

\newcommand{\RFlower}[4]{\raise -10pt\hbox{ \begin{picture}(35,35)
\put(24,26){\circle*{3}}\put(26,26){$\scriptscriptstyle #4$}\put(24,16){\line(0,1){10}}
    
\put(4,16){\circle*{3}}\put(-2,16){$\scriptscriptstyle #2$} \put(24,16){\circle*{3}} \put(26,16){$\scriptscriptstyle #3$}  
\put(14,6){\line(-1,1){10}} \put(14,6){\line(1,1){10}}
    
\put(14,6){\circle{3}} \put(14,-1){$\scriptscriptstyle #1$}
\end{picture}}}

\newcommand{\CFlower}[4]{\raise -10pt\hbox{ \begin{picture}(25,30)
\put(10,16){\circle*{3}}       
\put(10,6){\circle{3}}   \put(0,16){\circle*{3}}   \put(20,16){\circle*{3}}
\put(10,6){\line(-1,1){10}}  \put(10,6){\line(0,1){10}}  \put(10,6){\line(1,1){10}}
\put(10,-1){$\scriptscriptstyle #1$}   \put(-4,20){$\scriptscriptstyle #2$}   \put(10,20){$\scriptscriptstyle #3$} \put(20,20){$\scriptscriptstyle #4$}  
\end{picture}}
}

\newcommand{\SymGrp}[1]{\mathbb{S}_{#1}}

\DeclareMathOperator{\K}{\mathbb{K}}

\DeclareMathOperator{\homs}{Hom}
\newcommand{\Hom}[3]{\homs_{#1}(#2,#3)}

\DeclareMathOperator{\stab}{Stab}
\newcommand{\Stab}[2]{\stab_{#1}(#2)}

\author{Andrea Cesaro}
\title{\bf On $ PreLie $ algebras with divided symmetries}
\begin{document}

\maketitle
\begin{abstract}
We study an analogue of the notion of $ p $-restricted $ Lie $-algebra and of the notion of divided power algebra for $ PreLie $-algebras. We deduce our definitions from the general theory of operads. We consider two variants $\Lambda(P,-)$ and $\Gamma(P,-)$ of the monad $S(P,-)$ which governs the category of algebras classically associated to an operad $P$. For the operad $PreLie$ corresponding to $PreLie$-algebras, we prove that the category of algebras over the monad $\Lambda(PreLie,-)$ is identified with the category of $p$-restricted $PreLie$-algebras introduced by A. Dzhumadil'daev. We give an explicit description of the structure of an algebra over the monad  $\Gamma (PreLie,-)$ in terms of brace-type operations and we compute the relations between these generating operations. We prove that classical examples of  $PreLie$-algebras occurring in deformation theory actually form $\Gamma (PreLie,-)$-algebras. 
\end{abstract}
\section*{Introduction}
In this chapter, we study an analogue of the notion of $ p $-restricted $ Lie $-algebra and of the notion of divided power algebra for $ PreLie $-algebras.

$ PreLie $-algebras were introduced by Gerstenhaber in \cite{article:Gerstenhaber63} to encode structures related to the deformation complex of algebras. In recent years, applications of $ PreLie $-algebras appear in many other topics. Notably it has been discovered that they play a fundamental role in Connes-Kreimer's renormalization methods.

The category of $ PreLie $-algebras is associated to an operad denoted by $ PreLie $. To define our notion of $ PreLie $-algebras with divided symmetries, we use the general theory of B. Fresse \cite{article:Fresse00}, who showed how to associate a monad $ \Gamma(P,-) $ to any operad $P$ in order to encode this notion of algebra with divided symmetries. 

Recall that the usual category of algebras associated to an operad $P$ is governed by a monad $S(P,-)$ given by a generalized symmetric algebra functor with coefficients in the components of the operad $P$. To define $\Gamma(P,-)$ we merely replace the modules of coinvariant tensors, which occur in the generalized symmetric algebra construction, by modules of invariants. We denote by $\Lambda (P,-)$ the monad given by the image of the trace map between $ S(P,-) $ and $\Gamma(P,-)$. For short, we call $\Gamma P$-algebras the category of algebras governed by the monad $\Gamma(P,-)$, and we similarly call $ \Lambda P $-algebras the category of algebras governed by the monad $\Lambda(P,-)$. It turns out that many variants of algebra categories associated to these monads are governed by these monads. For instance, for the operad $P= Lie$ a $\Lambda Lie$-algebra is equivalent to a Lie algebra equipped with an alterned Lie bracket $[x,x]=0$, while the ordinary category of algebras over the operad $ Lie $ only depicts Lie algebras equipped with an antisymmetric Lie bracket $[x,y]=-[y,x]$ (which differs from the latter when the ground field has characteristic two). The category of $\Gamma Lie$-algebras, on the other hand, turns out to be equivalent to the classical notion of a $p$-restricted Lie algebra, where $p$ is the characteristic of the ground field (see \cite{article:Fresse00} and \cite{incollection:Fresse04})

We aim to give a description in terms of generating operations of the structure of an algebra over the monads $ \Lambda(PreLie,-) $ and $ \Gamma(PreLie,-) $. Our main motivations come from the applications of $PreLie$-algebras in deformation theory. We will see that significant examples of $PreLie$-algebras occurring in deformation theory are actually $\Gamma PreLie$-algebras.

To be explicit, recall that a $PreLie$-algebra is a module $V$ equipped with an operation $\{-,-\}:V\otimes V\longrightarrow V$ such that:
\[ \{\{x,y\},z\}-\{x,\{y,z\}\}=\{\{x,z\},y\}-\{x,\{z,y\}\}, \]
for all $x$, $y$, and $z$ in $V$.

First, we study the algebras over $ \Lambda(PreLie,-) $. We prove that these algebras are identified with the notion of $ p $-restricted $ PreLie $-algebras in the sense of \cite{article:Dzhumadil01}. A $ p $-restricted $ PreLie $-algebra is a $ PreLie $-algebra where the following relation is satisfied
\[\{\ldots\{x,\underbrace{y\},\ldots\} y}_{p} \}= \{x,\{\ldots\{\underbrace{y,y\}\ldots \}y}_{p}\}. \]
 Our result explicitly reads:

\begin{customthm}{A}[Theorem \ref{thmpRes}]
We assume that the ground ring $\mathbb{K}$ is a field of characteristic $p$. A $\Lambda PreLie$-algebra is equivalent to a $\mathbb{K}$-module $V$ equipped with an operation $\{-,-\}: V\otimes V\longrightarrow V$ satisfying the $PreLie$-relation and the $p$-restricted $PreLie$-algebra relation.
\end{customthm}

Let $ V $ be a free $\mathbb{K}$-module with a fixed basis. We prove the existence of an isomorphism of graded free $\mathbb{K}$-modules between $ S(PreLie,V) $ and $ \Gamma(PreLie,V) $. Using this isomorphism we express the composition morphism of the free algebra $ \Gamma(PreLie,V) $ and find a normal form for its elements. We combine these results to give a presentation of $ \Gamma(PreLie,-) $: this monad is determined  by $ n+1 $-fold polynomial ``corollas'' operations $ \lbrace -;\underbrace{-,\ldots,-}_{n}\rbrace_{r_{1},\ldots,r_{n}} $ of degree $ (1,r_{1},\ldots,r_{n}) $ and which satisfy some relations. We obtain the following theorem:

\begin{customthm}{B}[Theorem \ref{ThmRep}]
If $V$ is a free module over the ground ring $\mathbb{K}$, then providing the module $V$ with a $\Gamma PreLie$-algebra structure is equivalent to providing $V$ with a collection of polynomial maps 
\[\lbrace -;\underbrace{-,\ldots,-}_{n}\rbrace_{r_{1},\ldots,r_{n}}:V\times\underbrace{V\times\ldots\times V}_{n}\longrightarrow V, \]
where $r_{1},\ldots, r_{n}\in \mathbb{N}$ and which are linear in the first variable and the following relations hold: 
\begin{align*}
\tag{1}
\lbrace x;y_{1},\ldots, y_{n}\rbrace_{r_{\sigma(1)},\ldots,r_{\sigma(n)}} =\lbrace x;y_{\sigma^{-1}(1)},\ldots, y_{\sigma^{-1}(n)}\rbrace_{r_{1},\ldots,r_{n}}, 
\intertext{for any $ \sigma\in \SymGrp{n} $;}
\end{align*}

\vspace{-40pt}
\begin{multline*}
\tag{2}
\lbrace x;y_{1},\ldots,y_{i-1},y_{i},y_{i+1},\ldots, y_{n}\rbrace_{r_{1},\ldots, r_{i-1}, 0, r_{i+1},\ldots,r_{n}}=\\
\lbrace x;y_{1},\ldots,y_{i-1},y_{i+1},\ldots, y_{n}\rbrace_{r_{1},\ldots, r_{i-1}, r_{i+1},\ldots,r_{n}},
\end{multline*}

\begin{align*}
\tag{3}
\lbrace x;y_{1},\ldots,\lambda y_{i},\ldots, y_{n}\rbrace_{r_{1},\ldots, r_{i},\ldots,r_{n}}=\lambda^{r_{i}}\lbrace x;y_{1},\ldots, y_{i},\ldots, y_{n}\rbrace_{r_{1},\ldots, r_{i},\ldots,r_{n}},
\intertext{for any $\lambda$ in $\mathbb{K}$;}
\end{align*}
if $y_{i}=y_{i+1}$
\begin{multline*}
\tag{4}
\lbrace x;y_{1},\ldots, y_{i}, y_{i+1},\ldots, y_{n}\rbrace_{r_{1},\ldots, r_{i},r_{i+1},\ldots,r_{n}}=\\
\binom{r_{i}+r_{i+1}}{r_{i}} \lbrace x;y_{1},\ldots, y_{i}, y_{i+2},\ldots, y_{n}\rbrace_{r_{1},\ldots, r_{i}+r_{i+1},r_{i+2},\ldots,r_{n}}.
\end{multline*}

\begin{align*}
\tag{5}
\lbrace x; y_{1},\ldots,y_{i-1},a+b,y_{i+1},\ldots, y_{n}\rbrace_{r_{1},\ldots, r_{i},\ldots,r_{n}}=\sum_{s=0}^{r_{i}} \lbrace x; y_{1},\ldots,a,b,\ldots, y_{n}\rbrace_{r_{1},\ldots, s, r_{i}-s,\ldots,r_{n}},
\end{align*}

\begin{align*}
\tag{6}
\lbrace -;\rbrace = id,
\end{align*}

\vspace{-20pt}
\begin{multline*}
\tag{7}
\lbrace \lbrace x;y_{1},\ldots, y_{n} \rbrace_{r_{1},\ldots,r_{n}};z_{1},\ldots, z_{m} \rbrace_{s_{1},\ldots,s_{m}} = \\
\shoveright{\sum_{s_{i}=\beta_{i}+\sum\alpha_{i}^{\centerdot,\centerdot}}\dfrac{1}{\prod (r_{j}!)} \lbrace x;\lbrace y_{1};z_{1},\ldots, z_{m} \rbrace_{\alpha_{1}^{1,1},\ldots,\alpha_{m}^{1,1}},\ldots,\lbrace y_{1};z_{1},\ldots, z_{m} \rbrace_{\alpha_{1}^{1,r_{1}},\ldots,\alpha_{m}^{1,r_{1}}},} \\ 
\shoveright{\ldots, \lbrace y_{n};z_{1},\ldots, z_{m} \rbrace_{\alpha_{1}^{n,1},\ldots,\alpha_{m}^{n,1}},\ldots,\lbrace y_{n};z_{1},\ldots, z_{m} \rbrace_{\alpha_{1}^{n,r_{n}},\ldots,\alpha_{m}^{n,r_{n}}},} \\
z_{1},\ldots, z_{m} \rbrace_{1,\ldots, 1,\beta_{1},\ldots,\beta_{m}},
\end{multline*}
where, to give a sense to the latter formula, we use that the denominators $ r_{j}! $ divide the coefficient of the terms of the reduced expression which we get by applying relations (1) and (4) to simplify terms with repeated inputs on the right hand side (see Example \ref{xmp:Rel7}).
\end{customthm}

We give explicit examples of $ \Gamma PreLie $-algebras in the last section. Notably we explain that $\Gamma PreLie$-algebras naturally occur in the study of $ Brace $-algebras in characteristic different from $ 0 $. The already alluded to applications of $\Gamma PreLie$-algebras in deformation theory actually arise from this relationship.

\section*{Contents}
We devote sections $ 1-2 $ to general recollections on the operadic background of our constructions and to the definition of the operad $ PreLie $. 

In section $ 3 $ we construct a normal form for $ S(PreLie,V) $ and a basis for $ \Gamma(PreLie,V) $. We establish the equivalence between $ \Lambda PreLie $-algebras and $ p $-restricted $ PreLie $-algebras in section $ 4 $. We give the construction of a presentation of $ \Gamma(PreLie,-) $ in section $ 5 $. We conclude with examples of $ \Gamma PreLie $-algebras in section $ 6 $.

\section{Operads and their monads}

In this section, we briefly survey the general definitions of operad theory which we use in the chapter. This section does not contain any original result. We follow the presentation of \cite{article:Fresse00} for the definition of the monads $ \Lambda(P,-) $ and $ \Gamma(P,-) $ associated to an operad. We refer to the books \cite{book:MarklSchniderStasheff02}, \cite{book:LodayVallette12}, and \cite{book:Fresse09} for a comprehensive account of the theory of operads.
We work in a category of modules, $ \Mod $, over a fixed commutative ground ring $ \K $. For simplicity, we assume that $\K$ is a field in the  
statement of the general results and in the account of the general constructions of this section. We only consider the general case of a ring in concluding remarks at the end of each subsection (\ref{subsec:SMod}-\ref{subsec:NonSymOp}). We explain in these remarks the extra assumptions which we need to make our constructions work when we work over a ring.

\subsection{$ \SymGrp{} $-Modules}\label{subsec:SMod}
We recall the notion of $ \SymGrp{} $-module, which underlies the notion of operad, and the definition of three monoidal structures on the category of $ \SymGrp{} $-modules.
\begin{dfn}
We denote by $ \SymGrp{n} $ the symmetric group on a finite set of $ n $ elements. An $ \SymGrp{} $-module $ M $ is a collection $ \lbrace M(n) \rbrace_{n\in \mathbb{N}} $ where $M(n)$ is an $ \SymGrp{n} $-$ \mathbb{K} $-module for each $n\in \mathbb{N}$.

A morphism between $ \SymGrp{} $-modules $ f:M\longrightarrow N $ is a collection $ \lbrace f_{n}:M(n)\longrightarrow N(n)\rbrace_{n\in\mathbb{N}} $ where $f_{n}$ is a morphism of $ \SymGrp{n} $-$ \mathbb{K} $-modules for each $n\in \mathbb{N}$.

We denote by $ \ModSigma $ the category of $ \SymGrp{} $-modules.
\end{dfn}

We recall the definition of generalized symmetric (respectively, divided symmetries) tensors associated to an $ \SymGrp{} $-module. These functors determine the monads that define algebras over an operad and divided symmetries algebras over an operad.
\begin{dfn}
We denote by $ End(\Mod) $ the category of endofunctors of $ \Mod $. Let $ V $ be an $\mathbb{K}$-module and $M$ be an $\SymGrp{}$-module. On $ V^{\otimes n} $ the monoidal structure of the tensor product induces a natural $ \SymGrp{n} $-action. The $\K$-module $ M(n)\otimes V^{\otimes n} $ is equipped with the diagonal $ \SymGrp{n} $ action. The Schur functor $ S(M,-):\Mod\longrightarrow\Mod $ is defined as:
\begin{equation*}
 \begin{split}
 S(M,V)\eqdef\bigoplus_{n}M(n)\otimes_{\SymGrp{n}}V^{\otimes n},
 \end{split}
\end{equation*}
where $ \otimes_{\SymGrp{n}} $ means the  $\mathbb{K}$-module of co-invariants of the tensor product $ M(n)\otimes V^{\otimes n} $ under the diagonal action; and the coSchur functor $ \Gamma(M,-):\Mod\longrightarrow\Mod $ is defined as:
\begin{equation*}
 \begin{split}
 \Gamma(M,V)\eqdef\bigoplus_{n}M(n)\otimes^{\SymGrp{n}}V^{\otimes n},
 \end{split}
\end{equation*}
where $ \otimes^{\SymGrp{n}} $ means the  $\mathbb{K}$-module of invariants.

We then have two functors $ S:\ModSigma\longrightarrow End(\Mod) $ defined as $ S(M)\mapsto S(M,-) $, and  $ \Gamma:\ModSigma\longrightarrow End(\Mod) $ defined as $ \Gamma(M)\mapsto \Gamma(M,-) $.
\end{dfn}
\begin{rmk} 
The functors $ S(M,-) $ and $ \Gamma(M,-) $ are full and faithful when the ground ring is an infinite field. 
\end{rmk}

Between the coinvariant space and the invariant space there is a map called the trace (or norm) map.
\begin{dfn}
Let $M$ be an $ \SymGrp{} $-module. The trace map is the natural transformation $ Tr:S(M,-) \longrightarrow\Gamma (M,-) $ such that:
\begin{equation*}
 Tr(m\otimes v_{1}\otimes\ldots\otimes v_{n})=\sum_{\sigma\in\SymGrp{n}}\sigma^{\ast} (m\otimes v_{1}\otimes\ldots\otimes v_{n}),
\end{equation*}
for each $m\in M$, $ v_{1},\ldots v_{n}\in V $, and where we take the diagonal action of $\sigma\in \SymGrp{n}$ on the tensor $m\otimes v_{1}\otimes\ldots \otimes v_{n}\in M(n)\otimes V^{\otimes n}$.
\end{dfn}
\begin{rmk}\label{TraceIso}
The natural transformation $ Tr $ is an isomorphism in characteristic $ 0 $, but this is no longer the case in positive characteristic.
\end{rmk}
\begin{dfn}\label{DfnMonStr}
We consider three monoidal structures on $ \ModSigma $, let $ M $, $ N $ be two $ \SymGrp{} $-modules:
\begin{enumerate}
\item the tensor product $ -\boxtimes -:\ModSigma\times \ModSigma\longrightarrow \ModSigma $ of $ \SymGrp{} $-modules, defined by
\begin{equation*}
 \begin{split}
 (M\boxtimes N)(n)\eqdef \bigoplus_{i+j=n}\Ind^{\SymGrp{n}}_{\SymGrp{i}\times\SymGrp{j}}M(i)\otimes N(j),
 \end{split}
\end{equation*}
and whose unit is the $\SymGrp{}$-module such that $ \mathbb{K}(n)=\begin{cases} \mathbb{K} & \mbox{if } n = 0 \\ 
 0 & \mbox{if } n\neq 0, \\
 \end{cases} $
\item the coinvariant composition product $ -\utilde{\Box} -:\ModSigma\times \ModSigma\longrightarrow \ModSigma  $ of $ \SymGrp{} $-modules, defined by
\begin{equation*}
 \begin{split}
 M\utilde{\Box} N\eqdef \bigoplus_{r}M(r)\otimes_{\SymGrp{r}} N^{\boxtimes r},
 \end{split}
\end{equation*}
and whose unit is the $\SymGrp{}$-module such that $ I(n)=\begin{cases} \mathbb{K} &\mbox{if } n = 1 \\ 
 0 &\mbox{if } n\neq 1, \\
 \end{cases} $.
\item and the invariant composition product $ -\tilde{\Box} -:\ModSigma\times \ModSigma\longrightarrow \ModSigma  $ of $ \SymGrp{} $-modules, defined by
\begin{equation*}
 \begin{split}
 M\tilde{\Box} N\eqdef \bigoplus_{r}M(r)\otimes^{\SymGrp{r}} N^{\boxtimes r},
 \end{split}
\end{equation*}
with the same unit object as the coinvariant tensor product.
\end{enumerate} 
The tensor product $-\boxtimes -$ is symmetric, while the composition products $-\utilde{\Box} -$ and $-\tilde{\Box} -$ are not.
\end{dfn}

The two functors $ S $ and $ \Gamma $ are monoidal, more precisely:
\begin{prp}\label{ReflectingProducts}
The functors $ S:\ModSigma\longrightarrow End(\Mod) $ and $ \Gamma:\ModSigma\longrightarrow End(\Mod) $ define: 
\begin{enumerate}
\item strongly symmetric monoidal functors 

$ \xymatrix{
(\ModSigma,\boxtimes,\mathbb{K})\ar[r]^-{S} & (End(\Mod),\otimes,\mathbb{K})}$ and $\xymatrix{
(\ModSigma,\boxtimes,\mathbb{K})\ar[r]^-{\Gamma} & (End(\Mod),\otimes,\mathbb{K}),}$

where $ \otimes $ is the pointwise tensor product, inherited from the tensor product of  $\mathbb{K}$-modules, on the category of functors;
\item strongly monoidal functors 

$\xymatrix{
(\ModSigma,\utilde{\Box},\mathbb{K})\ar[r]^-{S} & (End(\Mod),\circ,Id)}$ and $\xymatrix{
(\ModSigma,\tilde{\Box},\mathbb{K})\ar[r]^-{\Gamma} & (End(\Mod),\circ,Id),}$

where $ \circ $ is the composition of functors.
\end{enumerate}
\end{prp}
\begin{proof}
These assertions are classical for $S$ (see for instance \cite[Ch. 5]{book:LodayVallette12}) and the analogue of these relations for $\Gamma$ is established in \cite{article:Fresse00}.
\end{proof}
\begin{rmk}\label{rmk:herring}
The statements of Proposition \ref{ReflectingProducts} remain valid without any change when we work with a commutative ground ring $\K$ in the case of the functor $S:M\mapsto S(M,-)$. 

For the functor $\Gamma(P,-)$ the statement of Proposition \ref{ReflectingProducts} is still valid if $\K$ is an hereditary ring, we restrict ourself to $\SymGrp{}$-modules $M$ whose components $M(r)$ are projective as $\K$-modules for all $r\in \mathbb{N}$, and we consider the restriction of our functor $\Gamma(M,-)$ to the category of projective $\K$-modules.

In short the tensor product $M(r)\otimes V^{\otimes r}$ form a projective $\K$-module as soon as $M(r)$ and $V$ do so. We just use the assumption that the ring $\K$ is hereditary to ensure that $M(r)\otimes^{\SymGrp{r}}V^{\otimes r} \subseteq M(r)\otimes V^{\otimes r}$ is still projective as a $\K$-module. We accordingly get that the map $\Gamma(M,-):V\mapsto \Gamma(M,V)$ defines an endofunctor of the category of projective $\K$-modules in this case. We then use that the tensor product with a projective module preserves kernels (and hence invariants) to check the validity of the claims of our proposition, after observing that the invariant composition of $\SymGrp{}$-modules also consists of projective $\K$-modules in this setting.
\end{rmk}
\subsection{Operads and $ P $-algebras}
We now recall the definition of an operad and the definition of the monads associated to an operad which we use in this chapter. To be specific, when we use the name operad, we mean symmetric operad, and we define this structure by using the coinvariant composition product recalled in the previous subsection.
\begin{dfn}
We define an operad to be a triple $ (P,\mu, \eta) $ where $ P $ is an $ \SymGrp{} $-module, $ 
 \mu:P\utilde{\Box} P\longrightarrow P, $ is a multiplication morphism, and $ \eta:I\longrightarrow P $ a unit morphism such that $P$ forms a monoid in $ (\ModSigma, \utilde{\Box}, I) $.
\end{dfn}

If $ P $ is an operad, then $ S(P,-) $ is a monad by Proposition \ref{ReflectingProducts}.

In what follows, we also use that the composition structure of an operad is determined by composition operations $\circ_{i}:P(m)\otimes P(n)\longrightarrow P(m+n-1)$  defined for any $ m,n\in \mathbb{N} $ and $1\leq i\leq m$, and which satisfies natural equivariance and associative relations. The unit morphism can then be given by a unit element $1 \in P(1)$ which satisfies natural unit relations with respect to these composition products. We refer to \cite{article:Fresse00} for instance for more details on this correspondence.

Since a general theory of free operads and their ideals can be set up (see \cite{book:Fresse09}) we can present operads by generating operations and relations.
\begin{dfn}
Let $ P $ be an operad, we define a $ P $-algebra to be an algebra over the monad $ S(P,-) $.
\end{dfn}

We have the following classical statement.
\begin{prp}\label{prp:OperadMonad}
Let $ V $ be a $\mathbb{K}$-module and $ (P,\mu,\eta) $ be an operad. The $\mathbb{K}$-module $ S(P,V) $ equipped with the morphisms induced by $ \mu $ and $ \eta $ is itself a $ P $-algebra.
\end{prp}
\begin{proof}
See \cite[Sec. 5.2.5]{book:LodayVallette12}.
\end{proof}
\begin{rmk}
The statement of \ref{prp:OperadMonad} remains valid without any extra assumption on our objects nor change when we work over a general ring $\K$.
\end{rmk}
\subsection{$ \Gamma (P,-) $ and $ \Lambda(P,-) $ monads}
Under a connectivity condition any operad structure on an $ \SymGrp{} $-module $ P $ induces a monad structure on $ \Gamma (P,-) $. We define $ \Gamma P $-algebras as the algebras for the monad $ \Gamma (P,-) $. The trace map is a natural transformation of monads. The concept of $ \Gamma P $-algebra was introduced by B. Fresse in \cite{article:Fresse00}. We recall the definition of these concepts in this section.
\begin{dfn}
An $ \SymGrp{} $-module $ N $ is connected if $ N(0)=0 $.
\end{dfn}

We rely on the following observation:
\begin{prp}\label{prp:TrMNIso}
Let $ M $ and $ N $ be two $ \SymGrp{} $-modules. If $ N $ is connected, then we have an isomorphism $ Tr_{M,N}:M\utilde{\Box} N\longrightarrow M\tilde{\Box}N $. 
\end{prp}
\begin{proof}
See \cite{article:Fresse00}.
\end{proof}

This proposition has the following consequence:
\begin{prp}
Let $ (P,\mu,\eta) $ be a connected operad. There exists a product $ \tilde{\mu}:P\tilde{\Box}P\longrightarrow P $ given by:
\begin{equation*}
 \begin{split}
 P\tilde{\Box}P \overset{\cong}{\longleftarrow} P\utilde{\Box} P \overset{\mu} {\longrightarrow} P
 \end{split}
\end{equation*}
and making $ (P,\tilde{\mu},\eta) $ into a monoid in the monoidal category $ (\ModSigma, \tilde{\Box}, I) $.\hfill$\square$
\end{prp}
\begin{crl}
Let $ (P,\mu,\eta) $ be a connected operad; then $ (\Gamma (P,-), \tilde{\mu}, \eta) $ is a monad.\hfill$\square$
\end{crl}
\begin{dfn}
Let $ (P,\mu,\eta) $ be a connected operad. A $ \Gamma P $-algebra is an algebra over the monad $ \Gamma (P,-) $.
\end{dfn}

From now on we only consider connected operads.

\begin{prp}
Let $P$ a connected operad. The natural transformation $Tr:S(P,-)\rightarrow \Gamma(P,-) $ is a morphism of monads.
\end{prp}
\begin{proof}
See \cite{article:Fresse00}.
\end{proof}

We introduce a third kind of algebras called $ \Lambda P $-algebras.
\begin{dfn}
We denote by $ \Lambda (P,-):\Mod\longrightarrow\Mod $ the functor defined by the epi-mono factorization of the trace map.
\end{dfn}

\begin{prp}
Let $P$ a connected operad. The functor $ \Lambda (P,-)$ forms a submonad of $ \Gamma(P,-) $ and the factorization 
\[S(P,-)\rightarrow \Lambda(P,-)\rightarrow \Gamma(P,-)\]
forms a monad morphism.
\end{prp}
\begin{proof}
We use that $Tr$ is a morphism of monads and that the functor $S(P,-)$ preserves the epimorphisms to obtain that we have a commutative diagram of the form:
\[
\xymatrix{
S(P,S(P,-))\ar@/^15pt/[rr]^{Tr \circ S(P, Tr)= \Gamma(P, Tr)\circ Tr} \ar@ {->>}[r] \ar[d] & \Lambda(P,\Lambda(P,-)) \ar@{-->}[d]^{\exists}\ar[r] & \Gamma(P,\Gamma(P,-))\ar[d]\\
S(P,-)\ar@/_15pt/[rr] _{Tr} \ar@ {->>}[r] & \Lambda(P,-)\ar@{^{(}->}[r] & \Gamma(P,-).}
\]
We deduce from this diagram that the composition product of the monad $\Gamma(P,-)$ and factor through $\Lambda(P,-)$. The unit of $\Gamma(P,-)$ similarly factors through $\Lambda(P,-)$. The conclusion of the proposition follows.
\end{proof}

\begin{dfn}
Let $ P $ be a connected operad, a $ \Lambda P $-algebra is an algebra for the monad $ \Lambda(P,-) $. 
\end{dfn}
\begin{rmk}
Any $ \Lambda P $-algebra $V$ is a $ P $-algebra. Any $ \Gamma P$-algebra $W$ is a $ \Lambda P $-algebra by the following commutative diagram:
\[
\xymatrix{
\Lambda (P,V)\ar[r]\ar[d] & V\\
\Gamma (P,V).\ar[ur]
}
\]
\end{rmk}

\begin{rmk}
The statements of this subsection have a generalization when we work over a hereditary ring. We then assume that the components of our operads $P(r)$ form projective $\K$-modules, for all $r\in \mathbb{N}$, and we use that the map $\Gamma(P,-):V\mapsto \Gamma(P,V)$ defines an endofunctor of the category of projective $\K$-modules, according to the observation of Remark \ref{rmk:herring}. We get that this functor $\Gamma(P,-)$ forms a monad in this case, and that $\Lambda(P,-)$ is a submonad of this monad over the category of projective $\K$-modules.

We can actually forget the assumption that $\K$ is hereditary in the case of the $PreLie$ operad which we study in the following section. We will actually see that $\Gamma(PreLie,-):V\mapsto \Gamma(PreLie,V)$ induces an endofunctor of the category of free $\K$-module without any further assumption on the ground ring $\K$.
\end{rmk}
\subsection{Non-symmetric operads and $ TP $-algebras}\label{subsec:NonSymOp}
We mostly use symmetric operads in this chapter. But we also consider a monad $ T(P,-) $ which is naturally associated to any non-symmetric operad. We explain this auxiliary construction in this subsection.
\begin{ntt}
We denote by $ \ModN $ the category of $\mathbb{K}$-modules graded on $ \mathbb{N} $.
\end{ntt}
\begin{dfn}
Let $ A $ be in $ \ModN $. There is a functor $ T(A,-):\Mod\longrightarrow\Mod $ defined as follows:
\begin{equation*}
\begin{split}
T(A,V)= \bigoplus_{n}A(n)\otimes V^{\otimes n}.
\end{split}
\end{equation*}
Forgetting the action of the symmetric groups we get a functor $U:\ModSigma\longrightarrow \ModN$. Composing $U$ with $T(-,-)$ we have a functor $ T:\ModSigma\longrightarrow End(\Mod) $.
\end{dfn}
\begin{dfn}
Let $ M $, $ N $ be two graded modules. We define the graded module $ M\Box N $ by:
\begin{equation*}
 \begin{split}
 M\Box N (n)= \bigoplus_{r}M(r)\otimes (\bigoplus_{n_{1}+\ldots n_{r}=n}N(n_{1})\otimes\ldots\otimes N(n_{r})),
 \end{split}
\end{equation*}
This operation gives a monoidal structure on $ \ModN $.
\end{dfn}

We have the following proposition:
\begin{prp}
The functor $ T:(\ModSigma,\Box, I)\longrightarrow (End(\Mod),\circ,Id) $ is strongly monoidal.
\end{prp}
\begin{proof}
We easily adapt the proof of the counterpart of this statement for $S$ and $\Gamma$.
\end{proof}
\begin{dfn}
Let $ P $ a non-symmetric operad, a $ TP $-algebra is an algebra over the monad $ T(P,-) $.
\end{dfn}

Let $ P $ be an operad and $ V $ a $\mathbb{K}$-module; then $ T(P,V) $ is a $ TP $-algebra with a structure map given by the map $ \mu $ on $ P $ and juxtaposition of words formed by elements of $ V $.
\begin{dfn}
There is a natural transformation given by the quotient 
\[ pr:T(P,-)\longrightarrow S(P,-).\]
\end{dfn}
\begin{prp} \label{RmkMonMor}
Let $P$ be a connected operad. The two natural transformations $ in $ and $ pr $ are monad morphisms.
\end{prp} 
\begin{proof}
Let $ V $ be a $\mathbb{K}$-module. This statement follows from the commutativity of the following diagrams:
\[\xymatrix{
T(P,T(P,V))\ar[r]\ar[d]_{\cong} & S(P,S(P,V))\ar[d]^{\cong}\\
T(P\Box P,V)\ar[d] & S(P\utilde{\Box} P,V)\\
T(P\utilde{\Box}P,V)\ar[ur]_{pr}.\\
}
\]
The verification of this commutative property is immediate.
\end{proof}

\begin{rmk}
The results of this subsection remain valid without change when we work over a commutative ring $\K$.
\end{rmk}

\section{On $ PreLie $ and rooted trees operads}
We recall the definition of $ PreLie $-algebras. These algebras have a binary product and a relation, sometimes called right associativity.

The $ PreLie $-algebras were introduced in \cite{article:Gerstenhaber63} by Gerstenhaber. We refer to \cite{article:ChapotonLivernet01} for the definition of the operad which governs this category of algebras. We also refer to \cite{incollection:Machon} for a survey on the theory and to \cite{article:Dokas13} for some applications of PreLie-algebras in positive characteristic.
\begin{dfn}\label{DfnPrelie}
A $\mathbb{K}$-module $ V $ is a $ PreLie $-algebra if it is endowed with a bilinear product:
\begin{equation*}
 \begin{split}
 \{-,-\}:V\otimes V\longrightarrow V,
 \end{split}
\end{equation*}
such that
\begin{equation*}\label{prelierel}
 \begin{split}
 \lbrace\lbrace x,y\rbrace ,z\rbrace -\lbrace x,\lbrace y,z\rbrace\rbrace = \lbrace\lbrace x,z\rbrace,y\rbrace -\lbrace\lbrace x,y\rbrace ,z\rbrace .
 \end{split}
\end{equation*}
\end{dfn}
The $ PreLie $ bracket defines a $ Lie  $ bracket by: $ [a,b]=\lbrace a,b\rbrace - \lbrace b,a\rbrace $ .

This structure appears naturally in different contexts. We recall some examples which we revisit in the context of $ \Gamma PreLie $-algebras.
\begin{xmp}\label{XmpPreLieAlg}
\begin{enumerate}
\item  Let $ P $ be an operad; we can define a $ PreLie $-algebra structure on the following $\mathbb{K}$-module $ \bigoplus_{n}P(n) $. Explicitly the $PreLie$-product is given by the following formula:
\[\lbrace p,q\rbrace=\sum_{i\in \lbrace 1,\ldots,n\rbrace}p\circ_{i}q\]
where $p\in P(n)$ and $q\in P(m)$. We go back to this example in Section \ref{Applications} where we study the relation between $PreLie$-systems and $\Gamma PreLie$-algebras.
\item The Hochschild complex of an associative algebra $ A $ defined as $ C^{r}(A,A) = \Hom{}{A^{\otimes r}}{A} $ has a dg-$ PreLie $-algebra structure. For $ f \in C^{m}(A,A) $ and $ g \in C^{n}(A,A) $, we explicitly have:
\begin{multline*}
 \lbrace f, g\rbrace (x_{1},\ldots,x_{n+m-1})=\\
 \sum_{i=1}^{m}(-1)^{(n-1)(i-1)}f(x_{1},\ldots,x_{i-1},g(x_{i},\ldots,x_{n+i-1}),x_{n-i}.\ldots,x_{n+m-1}),
\end{multline*}
This structure was introduced by Gerstenhaber in \cite{article:Gerstenhaber63} and can actually be defined on the deformation complex of any algebra over an operad (see \cite[Ch. 12]{book:LodayVallette12}). This $PreLie$-algebra structure on the Hochschild complex of an algebra is also a special case of the previous example, where we take $P=\Lambda End_{A}$, the operadic suspension $\Lambda$ of the endomorphism operad $End_{A}$ of $A$.
\end{enumerate}
\end{xmp}

We have a new type of $ PreLie $-algebras, called $ p $-restricted $ PreLie $-algebras, which occur when the ground ring is a field of characteristic $ p>0 $. As for $ p $-restricted $ Lie $-algebras, introduced by N. Jacobson in \cite{book:Jacobson79}, $ p $-restricted $ PreLie $-algebras appear naturally in the study of $ PreLie $ structures in positive characteristic $ p $. This kind of algebras was introduced by A. Dzhumadil'daev in \cite{article:Dzhumadil01}.
\begin{dfn}\label{dfn:pRest}
Fixed a field $ \mathbb{K} $ of characteristic $ p $. Let $ (L, \lbrace,\rbrace) $ be a $ PreLie $-algebra. It is a $ p $-restricted $ PreLie $, or $ p-PreLie $-algebra if the following equation holds:
\begin{equation*}\label{pRest}
\lbrace\lbrace\ldots\lbrace\lbrace x,\underbrace{ y \rbrace ,y\rbrace\ldots\rbrace y}_{p}\rbrace = 
\lbrace x,\lbrace\ldots\lbrace\lbrace\underbrace{ y  ,y\rbrace\ldots\rbrace y}_{p}\rbrace\rbrace.
\end{equation*}
\end{dfn}

\begin{rmk}
In \cite{article:Dokas13} I. Dokas introduces a more general notion of $p$-restricted $PreLie$-algebra. A ``generalized'' $p$-restricted $PreLie$-algebra is a $PreLie$-algebra $V$ endowed with a Frobenius map $\phi: V\rightarrow V$ satisfying some relations. If we assume  $\phi= \{\{\cdots\{\underbrace{y,y\},\cdots\}, y}_{p}\}$ we retrieve the definition of A. Dzhumadil'daev (Definition \ref{dfn:pRest}).
\end{rmk}

\begin{xmp}
\textbf{Simple Lie algebra $sl(2,\mathbb{K})$}. In characteristic $ 0 $ a semisimple Lie algebra does not admit a $ PreLie $ structure. But this is no longer the case in positive characteristic. In \cite{article:Dokas13} it is shown that $ sl(2,\mathbb{K}) $ admits a $ PreLie $ structure if and only if $ char(\mathbb{K})=3 $. In this case the $ PreLie $ structure is $ 3 $-restricted. For details and proof see \cite{article:Dokas13}.

\textbf{Rota-Baxter algebras}. In \cite{article:Dokas13} it is shown that the Rota-Baxter algebras, introduced by Gian-Carlo Rota in \cite{incollection:Rota95}, admit a $ p $-restricted $ PreLie $ structure.
\end{xmp}

$ PreLie $-algebras in the sense of \ref{DfnPrelie} are identified with a category of algebras over an operad defined by generators and relations. We recall another description of this operad in terms of trees.
\subsection{Non labelled trees}
In this section we introduce the definition of non labelled tree.
\begin{dfn}
We use the name non labelled tree to refer to a non-empty, finite, connected graph, without loops, with one special vertex called the root. The edges of such a tree admit a canonical orientation with the root as ultimate outgoing vertex, we have a pre-order corresponding to this orientation on the set of vertices of the tree, with the root as least element. Two non labelled trees are isomorphic if they are isomorphic as graphs by an isomorphism which preserves the root.

If necessary, we speak about a non labelled $ n $-tree to specify the number $n$ of vertices.
\end{dfn}
\begin{dfn}
Let $ \tau $ be a non labelled tree, a sub-tree is a connected sub-graph with root its minimum vertex by the pre-order defined by $ \tau $.
\end{dfn}
\begin{dfn}
Let $ \Tree $ be a non labelled rooted tree, a branch $ B $ of $ \Tree $ is a maximal subtree of $ \Tree $ that does not contain the root, where maximal has to be understood as a maximal element in the poset, defined by inclusion, of non labelled sub-trees of $ \Tree $.
\end{dfn}
\begin{dfn}
Let $ \Tree $ be a non labelled tree and $ B $ be a branch of $ \Tree $, the set $ iso(B) $ is the set of all branches of $ \Tree $ isomorphic, as non labelled trees, to $ B $.
\end{dfn}
\subsection{Labelled trees}
We define the concept of labelled tree.
\begin{dfn}
We call labelled tree, or just tree, a non labelled tree with a fixed bijection, called labelling, between its vertices and the set $ \lbrace 1,\ldots,n \rbrace $, where $ n $ is the number of vertices. We denote by $ \mathcal{RT}(n) $ the set of labelled trees with $ n $ vertices. The group $ \SymGrp{n} $ acts on this set by permuting the labelling.

If necessary, we use the expression of $ n $-tree to specify the number of vertices of a tree.
\end{dfn}
\begin{xmp}
The following is a $ 3 $-tree:
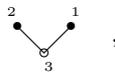
\begin{equation*}
 \FlowerTwo{3}{2}{1},
\end{equation*}
with root the vertex labelled by $ 3 $.

Notice that our trees are not planar. For example, we have:
\begin{equation*}
\FlowerTwo{3}{2}{1}=\FlowerTwo{3}{1}{2}.
\end{equation*}
\end{xmp}
\begin{dfn}
The $ \SymGrp{} $-module $ RT $ of rooted trees is
\begin{equation*}
 RT(n)\eqdef \mathbb{K} [\mathcal{RT}(n)],
\end{equation*}
where $ \mathbb{K}[X] $ is the $\mathbb{K}$-module freely generated by the base set $ X $.
\end{dfn}
\begin{xmp}
Let $ \sigma $ be the permutation of $ \SymGrp{3} $ that permutes $ 1 $ with $ 2 $ and fixes $ 3 $:
\begin{equation*}
 \sigma^{\ast}\FlowerTwo{1}{2}{3}=\FlowerTwo{2}{1}{3}.
\end{equation*}
\end{xmp}
\subsection{The rooted trees operad}
The $ \SymGrp{} $-module $ RT $ can be endowed with a structure of operad. This new operad is isomorphic to $ PreLie $. We review this result in this section.
The proof of the isomorphism is given in \cite{article:ChapotonLivernet01}. 
\begin{dfn}
We define the following partial compositions:
\begin{equation*}
 - \circ_{i} -:RT(m)\times RT(n)\longrightarrow RT(n+m-1),
\end{equation*}
with $ 1 \leq i\leq m $ as follows, let $In(\Tree,i)$ be the set of incoming edges of the vertex of $\Tree$ labelled $i$ :
\begin{equation*}
 \Tree \circ_{i} \upsilon \eqdef \sum_{f:In(\Tree,i)\longrightarrow \lbrace 1,\ldots, n \rbrace} \Tree\circ_{i}^{f} \upsilon ,
\end{equation*}
where $ \Tree\circ_{i}^{f} \upsilon $ is the $ n+m-1 $-tree obtained by substituting the tree $ \upsilon $ to the $i$th vertex of the tree $\tau$, by attaching the outgoing edge of this vertex in $\tau$, if it exists, to the root of $\upsilon$, and the ingoing edges to vertices of $\upsilon$ following the attaching map $f$ and then labelling following the labelling of $\Tree$ and the labelling of $\upsilon$ after obvious the shift. The sum runs over all these attachment maps $f:In(\tau, i)\longrightarrow \{1,\ldots n \}$.
\end{dfn}
\begin{xmp}
\begin{equation*}
 \Germ\circ_{1} \FlowerTwo{1}{2}{3}=\LFlower{1}{2}{3}{4}+\RFlower{1}{2}{3}{4}+\CFlower{1}{2}{3}{4}.
\end{equation*}
\end{xmp}
\begin{lmm}
These partial compositions define a total composition $ \gamma:RT\circ RT\longrightarrow RT $ that is an operad structure on the $ \SymGrp{} $-module $ RT $.
\end{lmm}
\begin{xmp}
\begin{equation*}
\Germ(\FlowerTwo{1}{2}{3},\Root)=\LFlower{1}{2}{3}{4}+\RFlower{1}{2}{3}{4}+\CFlower{1}{2}{3}{4}.
\end{equation*}
\end{xmp}
\begin{thm}[Chapoton, Livernet]
The $ PreLie $ operad is isomorphic to the $ RT $ operad. The isomorphism $ \varphi:PreLie\longrightarrow RT $ is realized by sending the generating operations of $ PreLie $ to $ \Germ $.
\end{thm}
\begin{proof}
See \cite{article:ChapotonLivernet01}.
\end{proof}
From now on we do not make any difference between $ RT $ and $ PreLie $ if it is not strictly necessary and therefore we will talk about trees as elements of $ PreLie $.
\section{A basis of $ \Gamma(PreLie,V) $}\label{GammaPreLiebasis}
The aim of this section is to make explicit a basis of the module $ \Gamma(PreLie, V) $ when $ V $ is a $ \mathbb{K} $-module equipped with a fixed basis $ \mathcal{V} $. 
\begin{dfn}\label{basisdfn}
Let $ x_{1},\ldots,x_{n} $ be elements of $ V $, and $ \Tree $ be an $ n $-tree. We denote the element $ \Tree\otimes x_{1}\otimes\ldots\otimes x_{n} $ in $ T(RT,V) $ by $ \Tree\langle x_{1},\ldots, x_{n}\rangle $ and the class $ [\Tree\otimes x_{1}\otimes\ldots\otimes x_{n}] $ in $ S(RT,V) $ by $ \Tree(x_{1},\ldots, x_{n})$ . 
If we fix a basis $ \mathcal{V} $ of $ V $, then we call:
\begin{itemize}
\item canonical basis of $ T(RT,V) $ the set $  \tbasis(\mathcal{RT},\mathcal{V})=\lbrace \Tree\langle x_{1},\ldots, x_{n}\rangle \vert \Tree\in \mathcal{RT}(n), x_{i}\in  \mathcal{V}\rbrace $, 
\item and canonical basis of $ S(RT,V) $ the set $  \basis(\mathcal{RT},\mathcal{V})=\lbrace \Tree( x_{1},\ldots, x_{n}) \vert \Tree\in \mathcal{RT}(n), x_{i}\in  \mathcal{V}\rbrace $.
\end{itemize}
The epimorphism $ pr:T(RT,V)\longrightarrow S(RT,V) $ restricts to a surjective function 
\[ pr:\tbasis(\mathcal{RT},\mathcal{V})\longrightarrow\basis(\mathcal{RT},\mathcal{V}).\]
\end{dfn}
\begin{dfn}\label{Stab}
Let $ V $ be a free $\mathbb{K}$-module with a fixed basis $ \mathcal{V} $. Let $ \tTree=\tau\langle x_{1}, \ldots x_{n}\rangle $ be an element of 
$ \tbasis(\mathcal{RT},\mathcal{V}) $. The stabilizer of $ \tTree $, denoted by $\Stab{}{\tTree}$, is the subgroup of $\SymGrp{n}$ defined by:
\begin{equation*}
 \begin{split}
 \Stab{}{\tTree}\eqdef \lbrace \sigma\in \SymGrp{n}\mid \sigma^{\ast}\tTree=\tTree\rbrace,
 \end{split}
\end{equation*}
where we consider the diagonal action of permutations $\sigma \in \SymGrp{n}$ on the tensor $\tau\otimes x_{1}\otimes\ldots\otimes x_{n}$ which represents our element $\tTree=\tau\langle x_{1},\ldots,x_{n}\rangle$.
\end{dfn}
\begin{xmp}
Let $ V $ be a free $\mathbb{K}$-module with a fixed basis $ \mathcal{V} $, and $ x,y,z $ be elements of $ \mathcal{V} $. We have the following formulas:
\begin{equation*}
 \begin{split}
 \Stab{}{\FlowerTwo{1}{2}{3}\lbrace x,y,z\rbrace} = \lbrace id \rbrace,
 \end{split}
\end{equation*}
and
\begin{equation*}
 \begin{split}
 \Stab{}{\FlowerTwo{1}{2}{3}\lbrace y,x,x\rbrace}=\lbrace id, (2,3)\rbrace.
 \end{split}
\end{equation*}
\end{xmp}
\begin{dfn}
We define $ F_{n} $ to be the following labelled $ n+1 $-tree $ \Flower{n+1} $.
\end{dfn}
\begin{prp}
Let $ V $ be a free $\mathbb{K}$-module with a fixed basis $ \mathcal{V} $ and $ x_{0},\ldots, x_{r} $ be elements of $ \mathcal{V} $ such that $ x_{p}\neq x_{q} $ if $ p\neq q $ and $ p,q\neq 0 $. Consider the element $ F_{n} \langle x_{0},\underbrace{x_{1},\ldots,x_{1}}_{i_{1}},\ldots,\underbrace{x_{r},\ldots,x_{r}}_{i_{r}}\rangle $ in $T(\mathcal{RT}, \mathcal{V})$; then $ \Stab{}{F_{n}\langle x_{0},\underbrace{x_{1},\ldots,x_{1}}_{i_{1}},\ldots,\underbrace{x_{r},\ldots,x_{r}}_{i_{r}}\rangle} $ is isomorphic to $ \SymGrp{i_{1}}\times\ldots\times\SymGrp{i_{r}} $.
\end{prp}
\begin{proof}
An element $ \sigma $ of $ \SymGrp{n+1} $ is in $ \Stab{}{F_{n}\langle x_{0},\underbrace{x_{1},\ldots,x_{1}}_{i_{1}},\ldots,\underbrace{x_{r},\ldots,x_{r}}_{i_{r}}\rangle} $ if its action fixes both $ F_{n} $ and $ x_{0},\underbrace{x_{1},\ldots,x_{1}}_{i_{1}},\ldots,\underbrace{x_{r},\ldots,x_{r}}_{i_{r}} $. Since $ F_{n} $ should be fixed we have that $ x_{0} $ is fixed and then $ \sigma $ has to be in the stabilizer of $ \underbrace{x_{1},\ldots,x_{1}}_{i_{1}},\ldots,\underbrace{x_{r},\ldots,x_{r}}_{i_{r}} $ that is isomorphic to $ \SymGrp{i_{1}}\times\ldots\times\SymGrp{i_{r}} $.
\end{proof}
\begin{dfn}\label{Def:Dec}
Let $ V $ be a free $\mathbb{K}$-module with a fixed basis $ \mathcal{V} $. Let $ \sTree $ be an element of $ \basis(\mathcal{RT},\mathcal{V}) $. We define $ \Dec(\sTree) $ to be the element of $ \basis(\mathcal{RT},\basis(\mathcal{RT},\mathcal{V})) $: 
\[ F_{r}(x_{0},B_{1},\ldots,B_{r}), \] 
where $ F_{r} $ is isomorphic as non labelled rooted tree to the full sub-corolla with root $ x_{0} $ the element of $ \mathcal{V} $ corresponding to the root of $ \sTree $, and $ B_{j} $ are elements of $ \basis(\mathcal{RT},\mathcal{V}) $ corresponding to the branches of $ \sTree $. 
\end{dfn}

In the literature the elements $F_{r}(x,T_{1},\ldots,T_{r})$ are sometimes denoted $B(x,T_{1},\ldots,T_{r})$, see \cite{article:ConnesKreimer98}.

\begin{xmp}
Let $\sTree$ be the element:
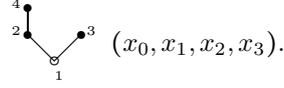
\[ \LFlower{1}{2}{3}{4}(x_{0},x_{1},x_{2},x_{3}). \]
We have:
\[ \Dec(\sTree)= F_{2}(x_{0},\Germ(x_{1}, x_{3}), \Root(x_{2})).\]
\end{xmp}

\begin{rmk}
Let $ V $ be a free $\mathbb{K}$-module with a fixed basis $ \mathcal{\mathcal{V}} $, and $ \sTree $ be an element of $ \basis(\mathcal{RT},\mathcal{V}) $. If $ \mu:S(RT,S(RT,V))\longrightarrow S(RT,V) $ is the composition product for the operad $ RT $ then $ \mu(\Dec(\sTree))=\sTree $.
\end{rmk}
\begin{dfn}
Let $ V $ be a free $\mathbb{K}$-module with a fixed basis $ \mathcal{\mathcal{V}} $. By iterating the process of Definition \ref{Def:Dec}, we can decompose any element $\sTree \in  S(\mathcal{RT},\mathcal{V}) $ into a composition of corollas whose roots are labelled by elements of the basis $\mathcal{V}$. We refer to this decomposition as the normal form of $\sTree$. It is unique up to the permutations of the non root entry of corollas.
\end{dfn}
\begin{dfn}
Let $ V $ be a free $\mathbb{K}$-module with a fixed basis $ \mathcal{V} $. Let $ \tTree $ be an element of $ \tbasis(\mathcal{RT},\mathcal{V}) $,  and $ \sigma $ be an element in $ \SymGrp{n} $. Then $ \Stab{}{\tTree} $ is isomorphic to $ \Stab{}{\sigma^{\ast}\tTree} $. Therefore we can define the group $ \Stab{}{\sTree} $ where $ \sTree $ is an element of $ \basis(\mathcal{RT},\mathcal{V}) $ as $ \Stab{}{\tTree} $ where $ \tTree $ is in the pre-image of $ \sTree $ under $ pr:\tbasis(\mathcal{RT},\mathcal{V})\longrightarrow \basis(\mathcal{RT},\mathcal{V}) $.
\end{dfn}
The group $ \Stab{}{\tTree} $ can be computed by induction.
\begin{prp}\label{PropStabDec}
Let $ V $ be a free $\mathbb{K}$-module with a fixed basis $ \mathcal{V} $, and $ \tTree $ be an element of $ \tbasis(\mathcal{RT},\mathcal{V}) $. Then $ \Stab{}{\tTree} $ is isomorphic to  $ \Stab{}{\Dec(\tTree)}\ltimes (\Stab{}{B_{1}}\times\ldots\times \Stab{}{B_{r})} $, where the semi-direct product is defined by the action of $ \Stab{}{\Dec(\tTree)} $ which permutes isomorphic branches.
\end{prp}
\begin{proof}
There is an obvious inclusion of $ \Stab{}{\Dec(\tTree)}\ltimes (\Stab{}{B_{1}}\times\ldots\times \Stab{}{B_{r}}) $ into $ \Stab{}{\tTree} $. Since any element of $ \Stab{}{\tTree} $ can be written in a unique way as a product of an element in $ \Stab{}{\Dec(\tTree)} $ and an element in $ (\Stab{}{B_{1}}\times\ldots\times \Stab{}{B_{r}}) $ the inclusion is actually an isomorphic.
\end{proof}
\begin{dfn}
Let $ V $ be a free $\mathbb{K}$-module with a fixed basis $ \mathcal{V} $. Let $ \tTree $ be an element of $ \tbasis(\mathcal{RT},\mathcal{V}) $. We set:
\begin{equation*}
\begin{split}
\gTree \eqdef \sum_{\sigma\in \SymGrp{n}/\Stab{}{\tTree}}  \sigma^{\ast}\tTree ,
\end{split}
\end{equation*}
\end{dfn}

\begin{rmk}
Let $ V $ be a free $\mathbb{K}$-module with a fixed basis $ \mathcal{V} $. Let $ \tTree $ be an element of $ \tbasis(\mathcal{RT},\mathcal{V}) $. We clearly have $ \gTree=\orb\sigma^{\ast}\tTree $ for any permutation $ \sigma $. Hence, this map passes to the quotient over coinvariants and induces a map of $\mathbb{K}$-modules $ \orb:S(RT,V)\longrightarrow \Gamma(RT,V) $ by linearity. Let $ x $ be an element of $ V $. It is easy to show that $ \orb \Root(x) $ is equal to $ \Root \otimes x $ (where $ \Root $ is the unique $ 1 $-tree). If there is no risk of confusion we will denote this element just by $ x $.
\end{rmk}

Notice that, in general, $ Tr(\tTree) $ differs from $ \gTree $.
\begin{xmp}\label{OrbVyxx}
Let $ V $ be a free $\mathbb{K}$-module with a fixed basis $ \mathcal{V} $, and $ x,y $ be elements of $ \mathcal{V} $, we compute $ \orb \FlowerTwo{1}{2}{3}\lbrace x,y,y \rbrace $: 
\begin{multline*}
\orb\FlowerTwo{1}{2}{3}(x,y,y)=\FlowerTwo{1}{2}{3}\otimes x \otimes y \otimes y + \FlowerTwo{2}{1}{3} \otimes y \otimes x \otimes y +\FlowerTwo{3}{2}{1} \otimes y \otimes y \otimes x.
\end{multline*}
\end{xmp}
Given a free $\mathbb{K}$-module $ V $, we want to compare the $\mathbb{K}$-modules $ \Gamma (RT,V) $ and $ S(RT,V) $. We show that they are isomorphic and that this isomorphism is realized by the map $ \orb:S(PreLie,V)\longrightarrow\Gamma (PreLie,V) $.

We use the following elementary result:

\begin{lmm}\label{lmmGset}
Let $ G $ be a group and $ X $ be a $ G $-set. There exists a isomorphism between $ \K[X]_{G} $ and $ \K[X]^{G} $, where $\K[X]$ is the free $\K$-module over the set $X$. \hfill $\square$
\end{lmm}
\begin{prp}\label{TrStabinj}
Let $ V $ be a free $\mathbb{K}$-module with a fix basis $ \mathcal{V} $. We define the set $ \gbasis(\mathcal{RT},\mathcal{V})=\lbrace \orb \sTree\vert \sTree\in \basis(\mathcal{RT},\mathcal{V})\rbrace $. The set $ \gbasis(\mathcal{RT},\mathcal{V}) $ forms a basis for the $\mathbb{K}$-module $ \Gamma (RT,V) $.
\end{prp}
\begin{proof}
By definition the map $ \orb:S(RT,V)\longrightarrow \Gamma (RT,V) $ gives a set-map $ \orb:\basis(\mathcal{RT},\mathcal{V})\longrightarrow \gbasis(\mathcal{RT},\mathcal{V}) $ defined by the bijection of Lemma (\ref{lmmGset}).
\end{proof}
\section{The equivalence between $ \Lambda PreLie $-algebras and $ p-PreLie $-algebras}
In this section we assume that $ \mathbb{K} $ is a field of positive characteristic $ p $ . We show that the categories of $ \Lambda PreLie $-algebras  and $ p-PreLie $-algebras are isomorphic. Let us observe that this implies that the category of $ p-PreLie $-algebras is a monadic subcategory of $ PreLie $-algebras.

In \cite{article:Dokas13} I. Dokas proves that $ \Gamma PreLie $-algebra are $ p $-restricted $ PreLie $-algebras. Here we improve this result by showing that the restricted $ PreLie $ structure is given by the $ \Lambda PreLie $ action on $ \Gamma PreLie $. 

\begin{rmk}
In \cite{article:Dokas13} I. Dokas introduces a more general notion of $p$-restricted $PreLie$-algebras, here we consider the less general definition given by in A. Dzhumadil'daev in \cite{article:Dzhumadil01}.
\end{rmk}

Recall that $ \Lambda (PreLie,V) $ is the target of the epimorphism given by the epi-mono decomposition of the trace map.
\[\xymatrix{
Ker(Tr_{V})\ar@{^{(}->}[dr]\\ 
& S(PreLie,V)\ar@{->>}[dr]\ar[rr]^{Tr_{V}} & &  \Gamma (PreLie,V)\\
& & \Lambda (PreLie,V) \ar@{^{(}->}[ur] & .}
\]
We compute the kernel of the trace map.
\begin{prp}\label{prpTrStab}
Let $ V $ be a $\mathbb{K}$-module with a fixed basis $ \mathcal{V} $. Let $ \sTree $ be an element of $ \basis(\mathcal{RT},\mathcal{V}) $. We have $ Tr(\sTree)=\vert \Stab{}{\sTree}\vert \orb\sTree $.
\end{prp}
\begin{proof}
Let $ \sTree $ be equal to $ \Tree(x_{1},\ldots, x_{n}) $ for some tree $ \Tree $ and $ x_{1},\ldots, x_{n} $ elements of $ \mathcal{V} $. Then the following equation holds:
\begin{multline*}
Tr(\Tree(x_{1},\ldots, x_{n}))=\sum_{\sigma\in \SymGrp{n}}\sigma^{\ast}\Tree \langle x_{1},\ldots, x_{n}\rangle=\\
\sum_{\Stab{}{\Tree( x_{1},\ldots, x_{n})}}\sum_{\sigma\in \frac{\SymGrp{n}}{\Stab{}{\Tree(x_{1},\ldots, x_{n})}}}\sigma^{\ast}\Tree\langle x_{1},\ldots, x_{n}\rangle=\vert \Stab{}{\sTree}\vert \orb \sTree.
\end{multline*}
\end{proof}
\begin{crl}
Let $ V $ be a $\mathbb{K}$-module with a fixed basis $ \mathcal{V} $. The kernel of the trace map is linearly generated by the elements $ \sTree $ of $ \basis(\mathcal{RT},\mathcal{V}) $ such that $ \vert \Stab{}{\sTree}\vert $ is a multiple of $ p $.
\end{crl}
\begin{proof}
The proof follows from Proposition \ref{prpTrStab} and from the observation that the map $\sTree\mapsto \orb\sTree$  defines a one-to-one correspondence from a basis of $S(RT,V)$ to a basis of $\Gamma (RT,V)$.
\end{proof}
\begin{lmm}\label{LmmGenKer}
Let $ V $ be a $\mathbb{K}$-module with a fixed basis $ \mathcal{V} $. Let $ \sTree $ be an element of $ \basis(\mathcal{RT},\mathcal{V}) $. Then $ \sTree $ has trace zero if and only if the expression $ F_{n+p}(x,\underbrace{B,\ldots, B}_{p}, B_{1},\ldots, B_{n}) $ with $ B $ and $ B_{i}\in \basis(\mathcal{RT},\mathcal{V})  $ appears in its normal form.
\end{lmm}
\begin{proof}
The proof follows from Proposition \ref{PropStabDec}, since $ \Stab{}{\sTree} $ is an iterated product of semi-direct products of symmetric groups representing the stabilizers of the corollas which compose the normal form of $ \sTree $.
\end{proof}

We improve this result and find a smaller collection of generators. First we fix the notation for multinomial coefficients.
\begin{ntt}\label{ntt:MultiBin}
Let $ k_{0} , \ldots , k_{r} $ be natural numbers and $ n=\sum_{i=0}^{r}k_{i} $. We define the multinomial coefficient $ (k_{0} , \ldots , k_{r}) $ to be $ \dfrac{n!}{k_{0}! \ldots k_{r}!} $.
\end{ntt}
\begin{lmm}\label{rmkCalcCorol}
Let $ V $ be a $\mathbb{K}$-module, $ x \in V $ and $B,B_{i}\in S(PreLie,V) $. The following equation holds:
\begin{gather*}
F_{p} (F_{n-p}(x,B_{1},\ldots,B_{n-p}),\underbrace{B,\ldots,B}_{p}) = F_{n} (x,\underbrace{B,\ldots,B}_{p},B_{1},\ldots,B_{n-p} )\\
+\sum_{\substack{i_{0}+\ldots +i_{v-p}=p\\ i_{0}<p}}(i_{0},\ldots ,i_{n-p}) F_{n-p+i_{0}}(x,\underbrace{B,\ldots, B}_{i_{0}},F_{i_{1}} (B_{1},\underbrace{B, \ldots, B}_{i_{1}}),\ldots , F_{i_{n-p}} (B_{n-p},\underbrace{B,\ldots, B}_{i_{n-p}}))
\end{gather*}  
in $ S(PreLie, S(PreLie(S(PreLie,V))) $.
\end{lmm}
\begin{proof}
Immediate consequence of the definition of composition of trees.
\end{proof}
For $g_{i}\in S(\mathcal{RT},(S(\mathcal{RT},\mathcal{V}))$ and $ f $ in $ \mathcal{RT}(n) $ we denote by $f(g_{1},\ldots,g_{n})$ the element in $ S(\mathcal{RT},S(\mathcal{RT},S(\mathcal{RT},\mathcal{V}))) $ representing their composite.
\begin{dfn}
Let $ V $ be a $\mathbb{K}$-module with a fixed basis $ \mathcal{V} $. A subset $G$ of $ \basis(\mathcal{RT},\basis(\mathcal{RT},\mathcal{V})) $ is said to generate $ Ker(Tr) $ if any element $ \sTree $ of $ \basis(\mathcal{RT},\mathcal{V}) $ is in $ Ker(Tr) $ if and only if it is the image by $\mu\circ \mu:S(RT,S(RT,S(RT,V)))\longrightarrow S(RT,V)$ of a linear combination of elements of the form $ f(g_{1},\ldots,g_{n}) $, where at least one $g_{i}$ is in $G$. 
\end{dfn}
\begin{lmm}\label{lmmGenKer}
Let $ V $ be a $\mathbb{K}$-module with a fixed basis $ \mathcal{V} $. The set $ K\eqdef \lbrace F_{p}(A,\underbrace{B,\ldots, B}_{p})\vert  A, B\in\basis(\mathcal{RT},\mathcal{V})   \rbrace $ generates $ Ker(Tr) $.
\end{lmm}
\begin{proof}
We compute $ F_{p} (F_{v}(x,B_{1},\ldots,B_{v}),\underbrace{B,\ldots,B}_{p}) $. By Lemma (\ref{rmkCalcCorol}) it is equal to:
\begin{gather*}
\sum_{i_{0}+\ldots +i_{v}=p}(i_{0},\ldots ,i_{v}) F_{v+i_{0}}(x,\underbrace{B,\ldots, B}_{i_{0}},F_{i_{1}} (B_{1},\underbrace{B, \ldots, B}_{i_{1}}),\ldots , F_{i_{v}} (B_{v},\underbrace{B,\ldots, B}_{i_{v}})),
\end{gather*}
We have that $ (i_{0},\ldots ,i_{v})= \binom {p} {i_{h}} k_{i_{h}} \forall 0\leq h \leq n$ for an integer $ k_{i_{h}} $ and so this coefficient differs from $ 0 $ modulo $ p $ if and only if $ i_{k}=0 $ for all but one index $ i_{k} $. Then we get a multinomial coefficient $ (0,\ldots,p,\ldots,0) = 1 $.
Therefore we have:
\begin{multline*}
 F_{p} (F_{v}(x,B_{1},\ldots,B_{v}),\underbrace{B,\ldots,B}_{p}) = F_{v+p} (x,\underbrace{B,\ldots,B}_{p},B_{1},\ldots,B_{v})\\
+ \sum_{i\in \lbrace 1,\ldots, v\rbrace} F_{v}  (x,B_{1},\ldots,B_{i-1},F_{p}(B_{i},\underbrace{B,\ldots,B}_{p}),B_{i+1},\ldots,B_{v}).
\end{multline*}

Let $ \sTree $ be an element of $ \basis(\mathcal{RT},\mathcal{V}) $. By Lemma \ref{LmmGenKer}, $ \sTree\in Ker(Tr) $ if and only if its normal form contains a corolla of the form $ g_{i}=F_{v+p}(x,\underbrace{B,\ldots, B}_{p}, B_{1}\ldots, B_{v}) $.  We use the above computation and the multi-linearity of a tree component to express this factor $ F_{v+p}(x,\underbrace{B,\ldots, B}_{p}, B_{1}, \ldots, B_{v}) $ as difference of terms of the form $f(g_{1},\ldots, g_{n})$ where at least one $g_{i}$ is in $K$. This proves the ``only if'' part of our claim.

To check the ``if'' part of our statement we use the above formula to express a factor $F_{p}(A,\underbrace{B,\ldots, B}_{p})$ with $\Dec(A)=F_{v}(x,B_{1},\ldots,B_{v})$ as sum of terms with either a factor 
\[F_{v+p} (x,\underbrace{B,\ldots,B}_{p},B_{1},\ldots,B_{v})\in Ker(Tr)\] 
or a factor 
\[F_{p}(B_{i},\underbrace{B,\ldots,B}_{p})\] 
where $B_{i}$  has strictly less vertices than $A$. Repeating the computation inductively of the equation and using the multi-linearity of the tree components we obtain, on the right side of the equation, a sum of elements in $ Ker(Tr) $. Since $Tr$ is a morphism of monads any $ f(g_{1},\ldots,g_{n}) \in S(RT,S(RT,S(RT,V))) $ such that at least one $g_{i}$ is in $K$ is in $Ker(Tr)$.
\end{proof}

The following definition appears in the literature with the name of heap order trees, see \cite{article:ConnesKreimer98}.
\begin{dfn}
Let $ \Tree $ be a non labelled tree. A labelling of vertices of $ \Tree $ is said to be an increasing labelling if it defines a total order refinement of the partial order on vertices induced by the tree, with the root as least element. We denote the number of possible increasing labellings by $ \lambda(\Tree) $.
\end{dfn}

\begin{xmp}
Consider the following non labelled tree:
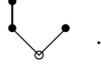
\begin{equation*}
\LFlower{}{}{}{}.
\end{equation*}
Then the following are the only three possible increasing labellings of $ \Tree $:
\begin{equation*}
\LFlower{1}{2}{3}{4},\LFlower{1}{2}{4}{3},\LFlower{1}{3}{2}{4},
\end{equation*}
and $ \lambda(\LFlower{}{}{}{})=3 $.
\end{xmp}
\begin{dfn}
We denote by $ n-ILTrees $ the set of $ n $-trees with an increasing labelling on vertices.
\end{dfn}

The following lemma is already treated in the literature using a different notation, see for example the operator $N$ in \cite{article:ConnesKreimer98},  the growth operator in \cite{article:Hoffman03}. It is used in the Butcher series for example in \cite{article:Brouder00} and in \cite{article:Livernet06}.
\begin{lmm}\label{lmmincrOrder}
In $ PreLie $ the following equation holds: 
\begin{equation*}
\varphi(\lbrace\lbrace\ldots\lbrace\lbrace \underbrace{-,- \rbrace , - \rbrace\ldots\rbrace -}_{n} \rbrace) =\sum_{\Tree\in n-ILTrees}\Tree(\underbrace{-,\ldots,-}_{n}),
\end{equation*}
where $ \varphi $ is the natural isomorphism between $ PreLie $ and $ RT $.
\end{lmm}
\begin{proof}
We proceed by induction on $ n $. If $ n $ is equal to $ 1 $ then the result is obviously true. Suppose the statement true at rank $ n-1 $. We then have:
\begin{gather*}
\varphi_{V}(\lbrace\lbrace\ldots\lbrace\lbrace y_{1},y_{2} \rbrace , y_{3}\rbrace\ldots\rbrace y_{n}\rbrace) =\\
\mu(\Germ(\sum_{\Tree\in (n-1)-ILTrees}\Tree(y_{1},\ldots , y_{n-1})), y_{n}))=\\
\sum_{\Tree\in (n-1)-ILTrees}\mu(\Germ(\Tree(y_{1},\ldots , y_{n-1})), y_{n})).
\end{gather*}
If $ \Tree $ is in $ (n-1)-ILTree $, then $ \Germ(\Tree(y_{1},\ldots , y_{n-1})), y_{n}) $ is the sum of $ n-1 $ distinct increasing labelling $ n $-trees. To any increasing labelling of an $ n $-tree is associated a labelling $ (n-1) $-tree obtained by dropping the leave labelled with $n$. We readily conclude that all the increasing labellings $ n $-tree appear once in the sum.
\end{proof}
\begin{prp}\label{prpincrOrder}
Let $ V $ be a $\mathbb{K}$-module with a fixed basis $ \mathcal{V} $ and $ x,y $ be elements of $ \mathcal{V} $. In $ S(PreLie,V) $, the following equation holds: 
\begin{equation*}
\varphi_{V}(\lbrace\lbrace\ldots\lbrace\lbrace x,\underbrace{ y \rbrace , y\rbrace\ldots\rbrace y}_{n}\rbrace) =\sum_{\Tree\in (n+1)-Trees}\lambda(\Tree)\Tree(x,y,\ldots , y),
\end{equation*}
where $ \lambda(\Tree) $ is the number of increasing labellings of $ \Tree $, and $ \varphi_{V} $ the natural isomorphism induced by the isomorphism of operads between $ PreLie $ and $ RT $.
\end{prp}
\begin{proof}
This identity follows from Lemma \ref{lmmincrOrder}.
\end{proof}
\begin{lmm}\label{prppRes}
Let $ V $ be a $\mathbb{K}$-module. Let $ x,y $ be elements of $ V $. The equation: 
\begin{equation*}
\lbrace\lbrace\ldots\lbrace\lbrace x,\underbrace{ y \rbrace ,y\rbrace\ldots\rbrace y}_{p}\rbrace = 
\lbrace x,\lbrace\ldots\lbrace\lbrace\underbrace{ y  ,y\rbrace\ldots\rbrace y}_{p}\rbrace\rbrace
\end{equation*}
holds in a $ PreLie $-algebra $ (V,\gamma) $ if and only if $ \gamma(F_{p}(x,\underbrace{y,\ldots, y}_{p}))=0 $.
\end{lmm}
\begin{proof}
By Proposition \ref{prpincrOrder} the left side of the equation can be expressed as a sum of trees with coefficients the number of possible increasing labellings. Let $ \Tree $ be a tree and $ B $ be a branch of $ \Tree $. That is:
\[ \Tree=F_{m_{B}+r}(\Root,\underbrace{B,\ldots, B}_{m_{B}},B_{1},\ldots, B_{r}), \]
where $B_{i}\ncong B$ for all $i\in\{1,\ldots, r\}$. We denote by the symbol $ n_{B} $ the number of vertices of $ B $ and by $S$ the tree
\[ F_{r}(\Root, B_{1},\ldots, B_{r}).\]
It is easy to check that the coefficient $ \lambda(\Tree) $ is equal to $ \binom {p}{n_{B}m_{B}} (\underbrace{n_{B},\ldots,n_{B}}_{m_{B}})\dfrac{1}{m_{B}!}\lambda(B)^{m_{B}}\lambda(S) $, where the symbol $(\underbrace{n_{B},\ldots,n_{B}}_{m_{B}})$ refers to the Notation \ref{ntt:MultiBin}.
Since $ p $ is a prime number $ p \nmid \lambda(\Tree)$ just in two cases:
\begin{enumerate}
\item $ n_{B}=p $, and $ m_{B}=1 $;
\item $ n_{B}=1 $, and $ m_{B}=p $.
\end{enumerate}
In the first case we obtain the following sum 
\begin{equation*}
\sum_{\Tree\in p-Trees}\lambda(\Tree)(x,\Tree(y,\ldots,y)).
\end{equation*}
But $ \lambda(\Tree)$ is equal to $\lambda(x,\Tree(y,\ldots,y)) $, and therefore applying $ \varphi^{-1} $ we get $  
\lbrace x,\lbrace\ldots\lbrace\lbrace\underbrace{ y  ,y\rbrace\ldots\rbrace y}_{p}\rbrace\rbrace $.
In the second case we obtain $ F_{p}(x,y,\ldots, y) $. This completes the proof.
\end{proof}
\begin{prp}\label{prp:pRestRel}
If $(V,\gamma)$ is a $ \Lambda PreLie $-algebra, then $ \lbrace -,-\rbrace:V \otimes V \longrightarrow V $ deduced from the $ PreLie $-algebra structure of $V$ is a $ p-PreLie $-algebra.
\end{prp}
\begin{proof}
By Lemma \ref{prppRes}  $(V,\gamma)$ satisfies the relation of $p$-restricted $PreLie$-algebra.
\end{proof}
\begin{thm}\label{thmpRes}
The construction of Proposition \ref{prp:pRestRel} gives an isomorphism between the categories of $ \Lambda PreLie $-algebras and of $ p-PreLie $-algebras.
\end{thm}
\begin{proof}
The category of $ \Lambda PreLie $-algebras is isomorphic to the subcategory of $ PreLie $-algebras $ (V,\gamma) $ such that the following diagram admits a factorization:
\[\xymatrix{
S(PreLie,V)\ar@{->>}[d]\ar[r]^-{\gamma} & V\\
\Lambda (PreLie,V)\ar@{-->}[ur] & .}
\]
This diagram admits an extension if and only if the composition 
\[ Ker(Tr)\longrightarrow S(PreLie,V)\longrightarrow V \]
is zero. By Lemma \ref{lmmGenKer} this is equivalent to say that $ (V,\gamma) $ is a $ p-PreLie $-algebra.
\end{proof}

\begin{prp}
The morphism $ \Gamma (Lie,-)\longrightarrow \Gamma (PreLie,-) $ factors through $ \Lambda (PreLie,-) $.
\end{prp}
\begin{proof}
To determine the image of the map it is enough to compute the image of a general bracket $ [-,-] $ and Frobenius power $ -^{[p]} $. They are respectively 
\[ \lbrace -,-\rbrace - (1,2)^{\ast}\lbrace -,-\rbrace \]
and 
\[ \underbrace{\lbrace\ldots\lbrace\lbrace -,-\rbrace,-\rbrace\ldots,-\rbrace}_{p}.\] 
These operations are contained in the sub-monad $ \Lambda (PreLie,-) $.
\end{proof}
\section{The $ \Gamma (PreLie,-) $ monad}
We go back to the case where $ \mathbb{K} $ is a commutative ring. We express the formula to compute the composition morphism of the monad $ \Gamma (PreLie,-) $. We use this formula to recover a normal form for the elements of $ \Gamma (PreLie,V) $.
\subsection{A formula for the $\Gamma (PreLie,-) $ composition}\label{GammaPreLieProduct}
Let $ V $ be a free $\mathbb{K}$-module, we show an explicit formula for the composition in $ \Gamma (PreLie,V) $. 

By Proposition \ref{TrStabinj} we have an explicit basis of $\Gamma(PreLie, V)$. So we compute the composition on it, and then extend by linearity. 
\begin{thm}\label{ThmProd}
Let $ V $ be a free $\mathbb{K}$-module with a fixed basis $ \mathcal{V} $. Let $\upsilon$ be an element of $\mathcal{RT}(n)$ and $\sTree_{1},\ldots ,\sTree_{n} $ be elements of $ \basis(\mathcal{RT},\mathcal{V}) $.  We assume that the composite of $\upsilon$ $\sTree_{1},\ldots ,\sTree_{n} $ in $ S(PreLie,\mathbb{Z}[\mathcal{V}]) $, where $\mathbb{Z}[\mathcal{V}]$ denotes the free $\mathbb{Z}$-module generated by $\mathcal{V}$, has the expansion:
\begin{equation*}
\mu( \upsilon( \sTree_{1},\ldots, \sTree_{n}))=\sum\limits_{\sTree\in\basis(\mathcal{RT},\mathcal{V})} \chi(\sTree)\sTree,
\end{equation*}
We then have the identity:
\begin{equation*}
\tilde{\mu}(\orb \upsilon(\orb \sTree_{1},\ldots, \orb \sTree_{n}))=\sum \dfrac{\chi(\sTree)\vert \Stab{}{\sTree}\vert}{\vert\Stab{}{\upsilon(\sTree_{1},\ldots ,\sTree_{n})}\vert \prod_{i}\vert\Stab{}{\sTree_{i}}\vert} \orb \sTree,
\end{equation*}
in $\Gamma(RT,V)$ where we consider the map $\tilde{\mu}:\Gamma(RT,\Gamma(RT,V))\rightarrow\Gamma(RT,V)$ and $\Stab{}{\upsilon(\sTree_{1},\ldots ,\sTree_{n})}$ is the stabilizer of $\upsilon(\sTree_{1},\ldots ,\sTree_{n})\in S(\mathcal{RT},S(\mathcal{RT},\mathcal{V}))$. (In the latter case we apply the definition of the stabilizer to the set $S(\mathcal{RT},\mathcal{W})$, where we take $\mathcal{W}=S(\mathcal{RT},\mathcal{V})$.)
\end{thm}
\begin{proof}
We start with a preliminary step. We use the notation $PreLie_{\K}$ to distinguish the coefficient in which is defined the operad $PreLie$. We consider the following morphism
\[\Gamma(PreLie_{\mathbb{Z}}, \mathbb{Z}[\mathcal{V}])\rightarrow \Gamma(PreLie_{\K}, \K[\mathcal{V}]),\]
induced by the canonical morphism $i_{PreLie}: PreLie_{\mathbb{Z}}\rightarrow PreLie_{\K}$ and $i_{V}:\mathbb{Z}[\mathcal{V}]\rightarrow \K[\mathcal{V}]$. We have the following commutative diagram:
\[
\xymatrix{
\Gamma(PreLie_{\mathbb{Z}},\Gamma(PreLie_{\mathbb{Z}},\mathbb{Z}[\mathcal{V}]))\ar[d]_{\cong}\ar[r] & \Gamma(PreLie_{\K},\Gamma(PreLie_{\K},\K[\mathcal{V}]))\ar[d]^{\cong}\\
\Gamma(PreLie_{\mathbb{Z}} \utilde{\Box} PreLie_{\mathbb{Z}},\mathbb{Z}[\mathcal{V}]))\ar[dd]_{\tilde{\mu}_{\mathbb{Z}}}\ar[r] & \Gamma(PreLie_{\K} \utilde{\Box} PreLie_{\K},\K[\mathcal{V}]))\ar[d]^{\cong}\\
& \Gamma( (PreLie \utilde{\Box} PreLie)_{\K},\K[\mathcal{V}]))\ar[d]^{\tilde{\mu}_{\K}}\\
\Gamma(PreLie_{\mathbb{Z}},\mathbb{Z}[\mathcal{V}])\ar[r] & \Gamma(PreLie_{\K},\K[\mathcal{V}])).}
\]

We assume $\K=\mathbb{Q}$ first and we check the relation in this case. We use the fact that $\orb \sTree = \dfrac{Tr(\sTree)}{\vert Stab(\sTree)\vert}$ and that $Tr:S(PreLie_{\mathbb{Q}},-)\rightarrow \Gamma(PreLie_{\mathbb{Q}},-)$ is an isomorphism of monads to get the identity:
\begin{align*}
{\mu}(Tr(\upsilon (Tr (\sTree_{1}),\ldots, Tr(\sTree_{n}))))&= Tr(\mu(\upsilon(\sTree_{1},\ldots, \sTree_{n})))\\
& = \sum\limits_{\sTree\in\basis(\mathcal{RT},\mathcal{V})} \chi(\sTree) Tr(\sTree)\\
& = \sum\limits_{\sTree\in\basis(\mathcal{RT},\mathcal{V})} \chi(\sTree) \vert Stab(\sTree) \vert \orb \sTree,
\end{align*}
and
\begin{align*}
{\mu}(Tr(\upsilon (Tr (\sTree_{1}),\ldots, Tr(\sTree_{n}))))&= \vert Stab(\upsilon(\sTree_{1},\ldots, \sTree_{n}) )\vert \prod_{i}\vert\Stab{}{\sTree_{i}}\vert \mu( \orb \upsilon(\orb \sTree_{1},\ldots, \orb \sTree_{n}))).\\
\end{align*}

We now consider the case $\K=\mathbb{Z}$. We have a monomorphism 
\[\Gamma(PreLie_{\mathbb{Z}}, \mathbb{Z}[\mathcal{V}])\hookrightarrow \Gamma(PreLie_{\mathbb{Q}}, \mathbb{Q}[\mathcal{V}])\] 
which respects the composition product. Thus the coefficients computed for the basis of the $\mathbb{Q}$-module $\Gamma(PreLie_{\mathbb{Q}}, \mathbb{Q}[\mathcal{V}])$ correspond to the coefficients for the basis of $\Gamma(PreLie_{\mathbb{Z}}, \mathbb{Z}[\mathcal{V}])$.

We consider the general case. The canonical morphism 
\[\Gamma(PreLie_{\mathbb{Z}},\mathbb{Z}[\mathcal{V}])\rightarrow \Gamma(PreLie_{\K},\K[\mathcal{V}])\] 
carries the relation, which is verified over $\mathbb{Z}$ for our basis elements, to the same relation over $\K$.
\end{proof}
\subsection{Decompositions in corollas and normal form}
Let $ V $ be a free $\mathbb{K}$-module with a fixed basis $ \mathcal{V} $. Let $ t $ be an element of $ \Gamma (PreLie,V) $; recall that by Proposition \ref{TrStabinj}, $ t $ is a linear combination of elements in $ \gbasis(\mathcal{RT},\mathcal{V}) $.

We present how to construct the elements of $ \Gamma (PreLie,V) $ from corollas.
\begin{lmm}\label{lmmgrowtrees}
Let $ V $ be a free $\mathbb{K}$-module with a fixed basis $ \mathcal{V} $. Let $ x $ be an element of $ \mathcal{V} $, and $ \sTree_{1},\ldots,\sTree_{r} $ be elements of $ \basis(\mathcal{RT},\mathcal{V}) $. Then 
\[ \tilde{\mu}(\orb F_{r}\lbrace x,\orb \sTree_{1},\ldots,\orb \sTree_{r}\rbrace) = \orb(\mu(F_{r}(x,\sTree_{1},\ldots, \sTree_{r}))). \]
\end{lmm}
\begin{proof}
The only thing to check is that the coefficient which appears in the left terms is one.
\end{proof}
\begin{lmm}\label{lmmDeczero}
Let $ V $ be a free $\mathbb{K}$-module with a fixed basis $ \mathcal{V} $. If $ \sTree $ is an element of $ \basis(\mathcal{RT},\mathcal{V}) $, then $ \orb \sTree $ is equal to $ \tilde{\mu}(\orb F_{r}(x_{a},\orb B_{1},\ldots,\orb B_{r})) $ where $  F_{r}[x_{a},B_{1},\ldots,B_{r}]  $ is $ Dec(\sTree) $.
\end{lmm}
\begin{proof}
We apply Lemma \ref{lmmgrowtrees} to $ Dec(\sTree) $.
\end{proof}
\begin{dfn}
Let $ V $ be a free $\mathbb{K}$-module with a fixed basis $ \mathcal{V} $, and $ \tTree $ be an element of $ \basis(\mathcal{RT},\mathcal{V}) $. We call normal form of $ \gTree $ its expression in iterated composition of elements of the form $ \orb(F_{r}(x,-,\ldots,-)) $ with $ x $ an element of $ \mathcal{V} $ deduced from the normal form of $ \tTree $.
\end{dfn}
\begin{prp}
Let $ V $ be a free $\mathbb{K}$-module with a fixed basis $ \mathcal{V} $. If $ \sTree $ is an element of $ \basis(\mathcal{RT},\mathcal{V}) $ then $ \orb \sTree $ admits a unique normal form.
\end{prp}
\begin{proof}
We apply Lemma \ref{lmmDeczero} recursively to get a bijection between the normal form in $ S(RT,V) $ and the normal form of $ \Gamma (RT,V) $.
\end{proof}
\begin{prp}
The set of monomials in normal form gives a basis of the $\mathbb{K}$-module $ \Gamma (RT,V) $.
\end{prp} 
\begin{proof}
It is easy to prove that the set of monomials in normal forms is in bijection with the set $ \basis(\mathcal{RT},\mathcal{V}) $. By Proposition \ref{TrStabinj} the set  $\gbasis(\mathcal{RT},\mathcal{V})$ forms a basis for $ \Gamma (RT,V) $ and it is in bijection with $ \basis(\mathcal{RT},\mathcal{V}) $.
\end{proof}
\subsection{A presentation for $ \Gamma (PreLie,-) $}
To describe the structure of $ \Gamma PreLie $-algebras we first show how to construct some polynomial abstract operations from a tree. We define a new type of algebras, $ Cor $-algebras, using just the abstract operations defined by corollas. We conclude the section by proving that $ Cor $-algebras coincide with $ \Gamma PreLie $-algebras.

Let $ (V,\gamma) $ be a $ \Gamma PreLie $-algebra and $ E_{n} $ be the free $\mathbb{K}$-module generated by a set of variables $ \mathcal{E}_{n}=\lbrace e_{i} \rbrace_{i\in \lbrace 1,\ldots,n\rbrace} $. We consider an element of $ \Gamma (RT,E_{n}) $. It can be written as a linear combination of elements of the form:
\[ \orb(\rho( \underbrace{e_{1}\otimes\ldots\otimes e_{1}}_{r_{1}}\otimes\ldots\otimes\underbrace{e_{n}\otimes\ldots\otimes e_{n}}_{r_{n}})), \] 
for some $ \rho\in \mathcal{RT} $. 

Let $ v_{1},\ldots,v_{n} $ be elements in $ V $, we define the morphism $ \psi_{v_{1},\ldots,v_{n}}:E_{n}\longrightarrow V $ by linear extension of $ \psi_{v_{1},\ldots,v_{n}}(e_{i})=v_{i} $. By functoriality it induces a morphism $ \psi_{v_{1},\ldots,v_{n}}:\Gamma(RT,E_{n})\longrightarrow\Gamma(RT,V) $.
\begin{dfn}
Any element $ \alpha $ of $ \Gamma(\mathcal{RT},\mathcal{E}_{n}) $ is of the form 
\[ \alpha=\orb(\rho(\underbrace{e_{1}\otimes\ldots\otimes e_{1}}_{r_{1}}\otimes\ldots\otimes\underbrace{e_{n}\otimes\ldots\otimes e_{n}}_{r_{n}})) \]
and induces a function 
\[ \varphi_{\underbrace{e_{1}\otimes\ldots\otimes e_{1}}_{r_{1}}\otimes\ldots\otimes\underbrace{e_{n}\otimes\ldots\otimes e_{n}}_{r_{n}}} :V^{\times n}\longrightarrow V \]
defined by:
\[ \varphi_{\underbrace{e_{1}\otimes\ldots\otimes e_{1}}_{r_{1}}\otimes\ldots\otimes\underbrace{e_{n}\otimes\ldots\otimes e_{n}}_{r_{n}}} (v_{1},\ldots,v_{n})=\gamma(\psi_{v_{1},\ldots,v_{n}}(\alpha)). \] 
The elements $ e_{i} $ have the role of abstract variables.
We denote the set of these functions by $ AbsOp_{n} $ and we set $ AbsOp=\coprod_{n\in \mathbb{N}}AbsOp_{n} $.
\end{dfn} 
\begin{dfn}
The group $ \SymGrp{n} $ acts on the set $ AbsOp_{n} $ by permutation of the indices $ \lbrace 1,\ldots, n\rbrace $. Let $ \sigma $ be an element of $\SymGrp{n} $. Let $ \varphi\in AbsOp $ be the element associated to
\[ \alpha=\orb(\rho(\underbrace{e_{1}\otimes\ldots\otimes e_{1}}_{r_{1}}\otimes\ldots\otimes\underbrace{e_{n}\otimes\ldots\otimes e_{n}}_{r_{n}})) \]
we define: 
\[ (\sigma_{\ast}\varphi)_{\underbrace{e_{1}\otimes\ldots\otimes e_{1}}_{r_{1}}\otimes\ldots\otimes\underbrace{e_{n}\otimes\ldots\otimes e_{n}}_{r_{n}}}\]
to be the element associated to
\[ \orb(\rho(\underbrace{e_{\sigma(1)}\otimes\ldots\otimes e_{\sigma(1)}}_{r_{1}}\otimes\ldots\otimes\underbrace{e_{\sigma(n)}\otimes\ldots\otimes e_{\sigma(n)}}_{r_{n}})) \]

From now on, we denote $ \varphi_{\underbrace{e_{1}\otimes\ldots\otimes e_{1}}_{r_{1}}\otimes\ldots\otimes\underbrace{e_{n}\otimes\ldots\otimes e_{n}}_{r_{n}}} $ by $ \varphi_{r_{1},\ldots , r_{n}} $ or  $ \varphi $.
\end{dfn}
\begin{prp}\label{prp:RelAn}
The following equations hold:
\begin{align}\label{Rel1}
\sigma_{\ast}\varphi(v_{1},\ldots,v_{n})=\varphi(v_{\sigma(1)},\ldots,v_{\sigma(n)}),
\end{align}
where $ \sigma\in \SymGrp{n} $,
\begin{align}
\varphi_{r_{1},\ldots,r_{i-1},0,r_{i+1},\ldots,r_{n}}(v_{1},\ldots, v_{n})=\varphi_{r_{1},\ldots,r_{i-1},r_{i+1},\ldots,r_{n}}(v_{1},\ldots,v_{i-1},v_{i+1}\ldots, v_{n}).\label{Rel2}
\end{align}
\begin{align}\label{Rel3}
\varphi_{r_{1},\ldots,r_{i}\ldots,r_{n}}(v_{1},\ldots,\lambda v_{i},\ldots, v_{n})=\lambda^{r_{i}}(\varphi_{r_{1},\ldots,r_{i},\ldots,r_{n}}(v_{1},\ldots, v_{i},\ldots, v_{n})).
\end{align}

If the function
\[ \varphi_{r_{1},\ldots,r_{i},r_{i+1},\ldots,r_{n}} \]
is commutative in the variables $ i $ and $ i+1 $ i.e. $ (i,i+1)^{\ast}\varphi=\varphi $, and $ v_{i}=v_{i+1} $, then
\begin{multline}\label{Rel4}
\varphi_{r_{1},\ldots,r_{i},r_{i+1},\ldots,r_{n}}(v_{1},\ldots,v_{i},v_{i+1},\ldots, v_{n})=\\
\binom{r_{i}+r_{i+1}}{r_{i}} \varphi_{r_{1},\ldots,r_{i-1}, r_{i}+r_{i+1},r_{i+2},\ldots,r_{n}}(v_{1},\ldots,v_{i-1}, v_{i},v_{i+2},\ldots, v_{n}).
\end{multline}

We have
\begin{gather}\label{Rel5}
\varphi_{r_{1},\ldots,r_{i}\ldots ,r_{n}}(v_{1},\ldots,\underset{i}{a+b},\ldots, v_{n})=\sum_{s=0}^{r_{i}} \varphi_{r_{1},\ldots, s,r_{i}-s,\ldots,r_{n}}(v_{1},\ldots, a,b,\ldots, v_{n}).
\end{gather}
where $v_{i}=a+b$
\end{prp}
\begin{proof}
These identities are immediate consequences of the multi-linearity of the operadic composition. 
\end{proof}
\begin{prp}\label{prp:RelAs}
Let $ \lbrace -;\underbrace{-,\ldots,-}_{n}\rbrace_{r_{1},\ldots,r_{n}} $ be the function defined by the corolla: 
\[ \orb F_{(\sum r_{\centerdot})}(e_{1},\underbrace{e_{2},\ldots,e_{2}}_{r_{1}},\ldots,\underbrace{e_{n+1},\ldots,e_{n+1}}_{r_{n}}). \]
We have the unit relation:
\begin{align}\label{Rel6}
\lbrace -;\rbrace = id,
\end{align}
and a distribution relation, which we formally write:
\begin{multline}\label{Rel7}
\lbrace \lbrace x;y_{1},\ldots, y_{n} \rbrace_{r_{1},\ldots,r_{n}};z_{1},\ldots, z_{m} \rbrace_{s_{1},\ldots,s_{m}} = \\
\shoveright{\sum_{s_{i}=\beta_{i}+\sum\alpha_{i}^{\centerdot,\centerdot}}\dfrac{1}{\prod (r_{j}!)} \lbrace x;\lbrace y_{1};z_{1},\ldots, z_{m} \rbrace_{\alpha_{1}^{1,1},\ldots,\alpha_{m}^{1,1}},\ldots,\lbrace y_{1};z_{1},\ldots, z_{m} \rbrace_{\alpha_{1}^{1,r_{1}},\ldots,\alpha_{m}^{1,r_{1}}},} \\ 
\shoveright{\ldots, \lbrace y_{n};z_{1},\ldots, z_{m} \rbrace_{\alpha_{1}^{n,1},\ldots,\alpha_{m}^{n,1}},\ldots,\lbrace y_{n};z_{1},\ldots, z_{m} \rbrace_{\alpha_{1}^{n,r_{n}},\ldots,\alpha_{m}^{n,r_{n}}},} \\
z_{1},\ldots, z_{m} \rbrace_{1,\ldots, 1,\beta_{1},\ldots,\beta_{m}},
\end{multline}
where, to give a sense to the latter formula, we use that the denominators $ r_{j}! $ divide the coefficient of the terms of the reduced expression which we get by applying relations (\ref{Rel1}) and (\ref{Rel4}) to simplify terms with repeated inputs on the right hand side.
\end{prp}
\begin{proof}
Let $\mathcal{V}$ be a basis of $V$ and $x,y_{1},\ldots, y_{n},z_{1},\ldots, z_{m}\in \mathcal{V}$ with possible repetition, the general case follows from relation (\ref{Rel3}) and relation (\ref{Rel5}). The Proposition is an immediate consequence of the formula of Theorem \ref{ThmProd}, where we take $\upsilon= F_{s_{1}+\ldots+ s_{m}}$, 
\[\sTree_{0}= F_{r_{1}+\ldots +r_{n}}(x,\underbrace {y_{1},\ldots, y_{1}}_{r_{1}},\ldots, \underbrace {y_{n},\ldots, y_{n}}_{r_{n}})\] 
which we plug into the root of $\upsilon$ and 
\begin{align*}
\sTree_{1} & = \Root (z_{1}), \ldots, \sTree_{s_{1}}= \Root (z_{1}),\ldots, \sTree_{s_{1}+\ldots + s_{m-1}+1}\\
& = \Root (z_{m}),\ldots ,\sTree_{s_{1}+\ldots + s_{m}}= \Root (z_{m}) 
\end{align*}
which we plug into the leaves of $\upsilon$. More precisely, the expansion of the composite is a linear combination of elements of the form $\orb \sTree$ where
\begin{align*}
\sTree& =F_{\sum b^{\bullet}_{\bullet}+\sum \beta_{\bullet}}(x,\\ 
&\quad \underbrace{F_{\alpha^{1,1}_{1}+\ldots +\alpha^{1,1}_{m}}(y_{1},\underbrace{z_{1},\ldots, z_{1}}_{\alpha^{1,1}_{1}},\ldots,\underbrace{z_{m},\ldots, z_{m}}_{\alpha^{1,1}_{m}}),\ldots,F_{\alpha^{1,1}_{1}+\ldots +\alpha^{1,1}_{m}}(y_{1},\underbrace{z_{1},\ldots, z_{1}}_{\alpha^{1,1}_{1}},\ldots,\underbrace{z_{m},\ldots, z_{m}}_{\alpha^{1,1}_{m}})}_{b^{1}_{1}},\ldots\\
&\quad \underbrace{F_{\alpha^{1,\gamma_{1}}_{1}+\ldots +\alpha^{1,\gamma_{1}}_{m}}(y_{1},\underbrace{z_{1},\ldots, z_{1}}_{\alpha^{1,\gamma_{1}}_{1}},\ldots,\underbrace{z_{m},\ldots, z_{m}}_{\alpha^{1,\gamma_{1}}_{m}}),\ldots,F_{\alpha^{1,\gamma_{1}}_{1}+\ldots +\alpha^{1,\gamma_{1}}_{m}}(y_{1},\underbrace{z_{1},\ldots, z_{1}}_{\alpha^{1,\gamma_{1}}_{1}},\ldots,\underbrace{z_{m},\ldots, z_{m}}_{\alpha^{1,\gamma_{1}}_{m}})}_{b^{1}_{\gamma_{1}}}\ldots\\
& \quad  \underbrace{F_{\alpha^{n,1}_{1}+\ldots +\alpha^{n,1}_{m}}(y_{n},\underbrace{z_{1},\ldots, z_{1}}_{\alpha^{n,1}_{1}},\ldots,\underbrace{z_{m},\ldots, z_{m}}_{\alpha^{n,1}_{m}}),\ldots,F_{\alpha^{n,1}_{1}+\ldots +\alpha^{1,1}_{m}}(y_{n},\underbrace{z_{1},\ldots, z_{1}}_{\alpha^{n,1}_{1}},\ldots,\underbrace{z_{m},\ldots, z_{m}}_{\alpha^{n,1}_{m}})}_{b^{n}_{1}},\ldots\\
& \quad \underbrace{F_{\alpha^{n,\gamma_{1}}_{1}+\ldots +\alpha^{1,1}_{m}}(y_{n},\underbrace{z_{1},\ldots, z_{1}}_{\alpha^{n,\gamma_{n}}_{1}},\ldots,\underbrace{z_{m},\ldots, z_{m}}_{\alpha^{n,\gamma_{n}}_{m}}),\ldots,F_{\alpha^{n,\gamma_{n}}_{1}+\ldots +\alpha^{n,\gamma_{n}}_{m}}(y_{1},\underbrace{z_{1},\ldots, z_{1}}_{\alpha^{n,\gamma_{n}}_{1}},\ldots,\underbrace{z_{m},\ldots, z_{m}}_{\alpha^{n,\gamma_{n}}_{m}})}_{b_{\gamma_{n}}^{n}}\ldots\\
&\quad  \underbrace{y_{1},\ldots,y_{1}}_{b^{1}_{\gamma_{1}+1}},\ldots, \underbrace{y_{n},\ldots,y_{n}}_{b^{n}_{\gamma_{n}+1}}, \underbrace{z_{1},\ldots,z_{1}}_{\beta_{1}},\ldots,\underbrace{z_{m},\ldots,z_{m}}_{\beta_{m}}),
\end{align*}
where $1\leq \gamma_{i}\leq r_{i}$ for all $1 \leq r \leq n$

We first compute the coefficient in front of $\orb \sTree$ by the formula of Theorem \ref{ThmProd}. We get 
\[\chi(\sTree)=\dfrac{\prod r_{\bullet}!}{\prod b^{\bullet}_{\bullet}!}\dfrac{\prod s_{\bullet}!}{\prod \alpha^{\bullet,\bullet}_{\bullet}! \prod \beta_{\bullet}!},\] 
\[\vert Stab(\upsilon(\sTree_{0},\ldots, \sTree_{s_{1}+\ldots \sTree{m}}))\vert= \prod s_{\bullet}!,\] 
and 
\[\vert \prod Stab(\sTree_{i})\vert= \prod r_{\bullet}!.\] 
If $y_{i}=z_{i}$ for all $i< t$, we have
\[\vert Stab(\sTree)\vert =\prod \alpha^{\bullet,\bullet}_{\bullet}!\prod b^{\bullet}_{\bullet}! \prod \beta_{\bullet}!\prod^{t}_{i=1} \binom{b^{i}_{\gamma_{i}+1}+ \beta_{i}}{\beta_{i}}.\] 
Thus the coefficient in front of $\orb \sTree$ is equal to $\prod^{t}_{i=1} \binom{b^{i}_{\gamma_{i}+1}+ \beta_{i}}{\beta_{i}}$. 

On the other hand in the relation (\ref{Rel7}) we first use relation (\ref{Rel1}) to sum all the terms associated to $\sTree$. We find the coefficients 
\[\dfrac{1}{\prod r_{\bullet}!}\dfrac{\prod r_{\bullet}!}{\prod b^{\bullet}_{\bullet}!}.\]
Then we apply relation (\ref{Rel4}) merging the common variables, hence we multiply the coefficient by
\[\prod b^{\bullet}_{\bullet}! \prod^{p}_{i=1} \binom{b^{i}_{\gamma_{i}+1}+ \beta_{i}}{\beta_{i}}.\] 
We obtain the same coefficient as before.
\end{proof}

\begin{xmp}\label{xmp:Rel7}
We have:
\begin{align*}
 \{\{x;y\}_{2}; y,z\}_{2,1}&= \dfrac{1}{2}(\{x;y,y,y,z\}_{1,1,2,1} + \{x;\{y;y\}_{1},\{y;y\}_{1},z\}_{1,1,1}\\
& \quad\quad + \{x;\{y;y\}_{1},y,y,z\}_{1,1,1,1} + \{x;y,\{y;y\}_{1},y,z\}_{1,1,1,1}\\
& \quad\quad + \{x;\{y;z\},y,y\}_{1,1,1}+ \{x;y,\{y;z\},y\}_{1,1,1}\\
& \quad\quad + \{x;\{y;y\},\{y;z\},y\}_{1,1,1}+ \{x;\{y;z\},\{y;y\},y\}_{1,1,1}\\
& \quad\quad + \{x;\{y;y,z\}_{1,1},y,y\}_{1,1,1}+ \{x;y,\{y;y,z\}_{1,1},y\}_{1,1,1}\\
& \quad\quad + \{x;\{y;y\}_{2},y,z\}_{1,1,1}+ \{x;y,\{y;y\}_{2},z\}_{1,1,1}\\
& \quad\quad + \{x;\{y;y\}_{2},\{y;z\}_{1}\}_{1,1}+ \{x;\{y;z\}_{1},\{y;y\}_{2}\}_{1,1}\\
& \quad\quad + \{x;\{y,y\}_{1},\{y;y,z\}_{1,1}\}_{1,1}+ \{x;\{y;y,z\}_{1,1},\{y,y\}_{1}\}_{1,1}\\
& \quad\quad + \{x;\{y;y,z\}_{2,1},y\}_{1,1}+ \{x;y,\{y;y,z\}_{2,1}\}_{1,1})\\
& \stackrel{(\ref{Rel1})}{=}\dfrac{1}{2}(\{x;y,y,y,z\}_{1,1,2,1} + \{x;\{y;y\}_{1},\{y;y\}_{1},z\}_{1,1,1}\\
& \quad\quad + 2\{x;\{y;y\}_{1},y,y,z\}_{1,1,1,1} + 2\{x;\{y;z\},y,y\}_{1,1,1}\\
& \quad\quad + 2\{x;\{y;y\},\{y;z\},y\}_{1,1,1} + 2\{x;\{y;y,z\}_{1,1},y,y\}_{1,1,1}\\
& \quad\quad + 2\{x;\{y;y\}_{2},y,z\}_{1,1,1} + 2\{x;\{y;y\}_{2},\{y;z\}_{1}\}_{1,1}\\
& \quad\quad + 2\{x;\{y,y\}_{1},\{y;y,z\}_{1,1}\}_{1,1} + 2\{x;\{y;y,z\}_{2,1},y\}_{1,1})\\
& \stackrel{(\ref{Rel4})}{=}\dfrac{1}{2}(12 \{x;y,z\}_{4,1} + 2 \{x;\{y;y\}_{1},z\}_{2,1}\\
& \quad\quad + 4\{x;\{y;y\}_{1},y,z\}_{1,2,1} + 4\{x;\{y;z\},y\}_{1,2}\\
& \quad\quad + 2\{x;\{y;y\},\{y;z\},y\}_{1,1,1} + 4\{x;\{y;y,z\}_{1,1},y\}_{1,2}\\
& \quad\quad + 2\{x;\{y;y\}_{2},y,z\}_{1,1,1} + 2\{x;\{y;y\}_{2},\{y;z\}_{1}\}_{1,1}\\
& \quad\quad + 2\{x;\{y,y\}_{1},\{y;y,z\}_{1,1}\}_{1,1} + 2\{x;\{y;y,z\}_{2,1},y\}_{1,1})\\
& = 6 \{x;y,z\}_{4,1} + \{x;\{y;y\}_{1},z\}_{2,1}\\
& \quad\quad + 2\{x;\{y;y\}_{1},y,z\}_{1,2,1} + 2\{x;\{y;z\},y\}_{1,2}\\
& \quad\quad + \{x;\{y;y\},\{y;z\},y\}_{1,1,1} + 2\{x;\{y;y,z\}_{1,1},y\}_{1,2}\\
& \quad\quad + \{x;\{y;y\}_{2},y,z\}_{1,1,1} + \{x;\{y;y\}_{2},\{y;z\}_{1}\}_{1,1}\\
& \quad\quad + \{x;\{y,y\}_{1},\{y;y,z\}_{1,1}\}_{1,1} + \{x;\{y;y,z\}_{2,1},y\}_{1,1},
\end{align*}
where we apply relation \ref{Rel1} to get our second identity and relation \ref{Rel4} to get our third identity.
\end{xmp}

\begin{dfn}
A $ Cor $-algebra is a $\mathbb{K}$-module $ V $ with a family of functions 
\[ \lbrace -;\underbrace{-,\ldots,-}_{n}\rbrace_{r_{1},\ldots,r_{n}}:V^{n+1}\longrightarrow V \] 
satisfying relations (\ref{Rel2}-\ref{Rel7}) as axioms. A morphism of $ Cor $-algebra is a linear map commuting with the operations $ \lbrace -;-,\ldots,-\rbrace_{r_{1},\ldots,r_{n}} $ i.e. 
\[ f(\lbrace -;-,\ldots,-\rbrace_{r_{1},\ldots,r_{n}})=\lbrace f(-);f(-),\ldots,f(-)\rbrace_{r_{1},\ldots,r_{n}}.\] 
\end{dfn}
\begin{prp}\label{prp:PLCor}
Let $ V $ be a $ \mathbb{K} $-module. A $ \Gamma PreLie $-algebra structure $ \gamma:\Gamma(PreLie,V)\longrightarrow V $ on $ V $ induces a natural $ Cor $-algebra structure on $ V $.
\end{prp}
\begin{proof}
We set $ \lbrace v;w_{1},\ldots,w_{n}\rbrace_{r_{1},\ldots,r_{n}}=\gamma(\orb F_{r_{1}+\ldots +r_{n}})(v,\underbrace{w_{1},\ldots,w_{1}}_{r_{1}},\ldots,\underbrace{w_{n},\ldots,w_{n}}_{r_{n}})) $. The statements of Propositions \ref{prp:RelAn} and \ref{prp:RelAs} show that it defines a $ Cor $-algebra.
\end{proof}

Our aim is to show that when we restrict to free $ \mathbb{K} $-modules the structures of $ \Gamma PreLie $-algebra and $ Cor $-algebra are equivalent. From now on, let $ V $ be a free $ \mathbb{K} $-module with a basis $ \mathcal{V}$ endowed with a $ Cor $-algebra structure. We aim to define a $ \Gamma PreLie $-algebra structure on $ V $ i.e. we define a morphism $ \gamma:\Gamma(RT,V)\longrightarrow V $ compatible with the action of $ \Gamma(PreLie,-) $ on $ \Gamma(PreLie,V) $.
\begin{cnt}\label{cnt:structuremap}
We set 
\[ \gamma(\orb (F_{(\sum r_{\centerdot})}(x,\underbrace{y_{1},\ldots,y_{1}}_{r_{1}},\ldots,\underbrace{y_{n},\ldots,y_{n}}_{r_{n}})))=\lbrace x;y_{1},\ldots, y_{n} \rbrace_{r_{1},\ldots,r_{n}} \] 
where $ x,y_{1},\ldots, y_{n}\in V $. By the normal form any element of $ \Gamma(\mathcal{RT},\mathcal{V}) $ can be decomposed in the iterated composition of corollas, the morphism $ \gamma $ is defined on the basis by composition of the function associated to corollas and then computed iteratively. 
\end{cnt}
\begin{lmm}
Let $ V $ be a free $ \mathbb{K} $-module with a basis $ \mathcal{V} $. If $ V $ is a $ Cor $-algebra then the assignment of construction \ref{cnt:structuremap} is well defined and does not depend on the choice of the basis $ \mathcal{V} $ of the $ \mathbb{K} $-module $ V $.
\end{lmm}
\begin{proof}
This follows from the relations (\ref{Rel1}), (\ref{Rel2}), (\ref{Rel3}), (\ref{Rel5}). More precisely given two basis $ \mathcal{V} $, $ \mathcal{W} $ of $ V $, we check  that the two maps $ \gamma_{\mathcal{V}},\gamma_{\mathcal{W}}:\Gamma(PreLie,V)\longrightarrow V $ are equal.
Let $ \tTree$ be a general element of $ \Gamma(\mathcal{RT},\mathcal{V}) $ such that $ \Dec(\tTree)=\orb F_{m}(v;\underbrace{\mathfrak{q}_{1},\ldots,\mathfrak{q}_{1}}_{t_{1}},\ldots, \underbrace{\mathfrak{q}_{r},\ldots,\mathfrak{q}_{r}}_{t_{r}}) $, for some $ \mathfrak{q}_{j}\in \Gamma(\mathcal{RT},\mathcal{V}) $, $ v\in \mathcal{V} $. Let $ \sum\limits_{j_{s}\in J_{s}}\lambda_{j_{s}}\mathfrak{p}^{j_{s}}_{s} $ be the linear decomposition of $ \mathfrak{q}_{s} $ in the basis $ \Gamma(\mathcal{RT},\mathcal{W}) $ and $ v=\sum\limits_{w_{i}\in \mathcal{W}}\xi_{i}w_{i} $. We proceed by induction on $ n $, the number of corollas appearing in the normal form of $ \tTree $. If $ n $ is equal to $ 0 $ then $ \tTree $ is the identity. Suppose the statement true at rank $ n-1 $. We then have by definition:
\begin{multline*} 
\Dec(\tTree)=\orb F_{m}(v;\underbrace{\mathfrak{q}_{1},\ldots,\mathfrak{q}_{1}}_{t_{1}},\ldots, \underbrace{\mathfrak{q}_{r},\ldots,\mathfrak{q}_{r}}_{t_{r}})=\\
\shoveright{\orb F_{m}(v;\underbrace{\sum\limits_{j_{1}\in J_{1}}\lambda_{j_{1}}\mathfrak{p}^{j_{1}}_{1},\ldots\sum\limits_{j_{1}\in J_{1}}\lambda_{j_{1}}\mathfrak{p}^{j_{1}}_{1}}_{t_{1}},\ldots, \underbrace{\sum\limits_{j_{r}\in J_{r}}\lambda_{j_{r}}\mathfrak{p}^{j_{r}}_{r},\ldots\sum\limits_{j_{r}\in J_{r}}\lambda_{j_{r}}\mathfrak{p}^{j_{r}}_{r}}_{t_{r}})=}\\ 
\shoveright{\sum\limits_{\substack{j^{h}_{k}\in J_{k}\\ \sum\limits_{u=1}^{a_{k}} s_{j^{u}_{k}} = t_{k}}}\lambda^{s_{j^{1}_{1}}}_{j^{1}_{1}}\ldots\lambda^{s_{j^{a_{r}}_{r}}}_{j^{a_{r}}_{r}} \vert \Stab{}{F}:\SymGrp{s_{j^{1}_{1}}}\times\ldots\times\SymGrp{s_{j^{a_{r}}_{r}}} \vert \orb F_{m}(v;\underbrace{\mathfrak{p}^{j^{1}_{1}}_{1},\ldots,\mathfrak{p}^{j^{1}_{1}}_{1}}_{s_{j^{1}_{1}}},}\\
\shoveright{\ldots,\underbrace{\mathfrak{p}^{j^{a_{1}}_{1}}_{1},\ldots,\mathfrak{p}^{j^{a_{1}}_{1}}_{1}}_{s_{j^{a_{1}}_{1}}}, \ldots, \underbrace{\mathfrak{p}^{j^{1}_{r}}_{r},\ldots,\mathfrak{p}^{j^{1}_{r}}_{r}}_{s_{j^{1}_{r}}},\ldots,\underbrace{\mathfrak{p}^{j^{a_{r}}_{r}}_{r},\ldots,\mathfrak{p}^{j^{a_{r}}_{r}}_{r}}_{s_{j^{a_{r}}_{r}}})=}\\
\shoveright{\sum\limits_{w_{i}\in \mathcal{W}} \xi_{i}\sum\limits_{\substack{j^{h}_{k}\in J_{k}\\ \sum\limits_{u=1}^{a_{k}} s_{j^{u}_{k}} = t_{k}}}\lambda^{s_{j^{1}_{1}}}_{j^{1}_{1}}\ldots\lambda^{s_{j^{a_{r}}_{r}}}_{j^{a_{r}}_{r}}\vert \Stab{}{C_{s_{j^{1}_{1}},\ldots,s_{j^{a_{1}}_{1}}})}:\SymGrp{s_{j^{1}_{1}}}\times\ldots\times\SymGrp{s_{j^{a_{r}}_{r}}} \vert\orb F_{m}(w_{i};\underbrace{\mathfrak{p}^{j^{1}_{1}}_{1},\ldots,\mathfrak{p}^{j^{1}_{1}}_{1}}_{s_{j^{1}_{1}}},} \\
\ldots,\underbrace{\mathfrak{p}^{j^{a_{1}}_{1}}_{1},\ldots,\mathfrak{p}^{j^{a_{1}}_{1}}_{1}}_{s_{j^{a_{1}}_{1}}}, \ldots, \underbrace{\mathfrak{p}^{j^{1}_{r}}_{r},\ldots,\mathfrak{p}^{j^{1}_{r}}_{r}}_{s_{j^{1}_{r}}},\ldots,\underbrace{\mathfrak{p}^{j^{a_{r}}_{r}}_{r},\ldots,\mathfrak{p}^{j^{a_{r}}_{r}}_{r}}_{s_{j^{a_{r}}_{r}}})
\end{multline*}
where $ \Stab{}{C_{s_{j^{1}_{1}},\ldots,s_{j^{a_{1}}_{1}}})} $ is the group 
\[ \Stab{}{\orb F_{m}(v;\underbrace{\mathfrak{p}^{j^{1}_{1}}_{1},\ldots,\mathfrak{p}^{j^{1}_{1}}_{1}}_{s_{j^{1}_{1}}},\ldots,\underbrace{\mathfrak{p}^{j^{a_{1}}_{1}}_{1},\ldots,\mathfrak{p}^{j^{a_{1}}_{1}}_{1}}_{s_{j^{a_{1}}_{1}}})}.\]
We have:
\begin{multline*} 
\gamma_{\mathcal{W}}(\tTree)=\\
\shoveright{\sum\limits_{w_{i}\in \mathcal{W}} \xi_{i}\sum\limits_{\substack{j^{h}_{k}\in J_{k}\\ \sum\limits_{u=1}^{a_{k}} s_{j^{u}_{k}} = t_{k}}}\lambda^{s_{j^{1}_{1}}}_{j^{1}_{1}}\ldots\lambda^{s_{j^{a_{r}}_{r}}}_{j^{a_{r}}_{r}} \vert \Stab{}{F}:\SymGrp{s_{j^{1}_{1}}}\times\ldots\times\SymGrp{s_{j^{a_{r}}_{r}}} \vert\lbrace w_{i};\gamma_{\mathcal{W}}(\mathfrak{p}^{j^{1}_{1}}_{1}),\ldots,\gamma_{\mathcal{W}}(\mathfrak{p}^{j^{a_{1}}_{1}}_{1}),}\\
\shoveright{\ldots,\gamma_{\mathcal{W}}(\mathfrak{p}^{j^{1}_{r}}_{r}),\ldots,\gamma_{\mathcal{W}}(\mathfrak{p}^{j^{a_{r}}_{r}}_{r}) \rbrace_{s_{j^{1}_{1}},\ldots,s_{j^{a_{r}}_{r}}}=} \\ 
\shoveright{\sum\limits_{\substack{j^{h}_{k}\in J_{k}\\ \sum\limits_{u=1}^{a_{k}} s_{j^{u}_{k}} = t_{k}}}\lambda^{s_{j^{1}_{1}}}_{j^{1}_{1}}\ldots\lambda^{s_{j^{a_{r}}_{r}}}_{j^{a_{r}}_{r}} \vert \Stab{}{F}:\SymGrp{s_{j^{1}_{1}}}\times\ldots\times\SymGrp{s_{j^{a_{r}}_{r}}} \vert \lbrace v;\gamma_{\mathcal{W}}(\mathfrak{p}^{j^{1}_{1}}_{1}),\ldots,\gamma_{\mathcal{W}}(\mathfrak{p}^{j^{a_{1}}_{1}}_{1}),}\\
\ldots,\gamma_{\mathcal{W}}(\mathfrak{p}^{j^{1}_{r}}_{r}),\ldots,\gamma_{\mathcal{W}}(\mathfrak{p}^{j^{a_{r}}_{r}}_{r}) \rbrace_{s_{j^{1}_{1}},\ldots,s_{j^{a_{r}}_{r}}}.
\end{multline*} 
Applying the $ Cor $-algebra relations (\ref{Rel4})-(\ref{Rel6}) we have:
\[ \gamma_{\mathcal{W}}(\tTree)=\lbrace v;\gamma_{\mathcal{W}}(\mathfrak{p}_{1}),\ldots, \gamma_{\mathcal{W}}(\mathfrak{p}_{r})\rbrace_{t_{1},\ldots,t_{r}}  \]
by induction hypothesis
\[ \gamma_{\mathcal{W}}(\tTree)=\lbrace v;\gamma_{\mathcal{W}}(\mathfrak{p}_{1}),\ldots, \gamma_{\mathcal{W}}(\mathfrak{p}_{r})\rbrace_{t_{1},\ldots,t_{r}}=\lbrace v;\gamma_{\mathcal{V}}(\mathfrak{p}_{1}),\ldots, \gamma_{\mathcal{V}}(\mathfrak{p}_{r})\rbrace_{t_{1},\ldots,t_{r}}= \gamma_{\mathcal{V}}(\tTree).\]
\end{proof}
\begin{dfn}
Let $ V $ be a free $ \mathbb{K} $-module with a basis $ \mathcal{V} $. Let $ \tTree $ be an element of $ \mathcal{RT} $ and $ w_{1},\ldots,w_{m}\in \mathcal{V} $. We say that an element of $ \Gamma(\mathcal{RT},\Gamma(\mathcal{RT}, \mathcal{V})) $ is simple if it is of the form $ \orb(\tTree(\Root (w_{1}),\ldots,\Root (w_{m}))) $ . 
\end{dfn}
\begin{lmm}
Let $ V $ be a free $ \mathbb{K} $-module with a basis $ \mathcal{V} $. Recall that $ \Root $ is the unique $ 1 $-tree. If $ \orb(\tTree(\Root (w_{1}),\ldots,\Root (w_{m}))) $ is a simple element then \[ \tilde{\mu}(\orb(\tTree(\Root (w_{1}),\ldots,\Root (w_{m}))))=\orb(\tTree(w_{1},\ldots,w_{m})).\]
\end{lmm}
\begin{lmm}
The Construction \ref{cnt:structuremap} is compatible with unit and composition in $ \Gamma(PreLie,V) $.
\end{lmm}
\begin{proof}
This follows from the relations (\ref{Rel4}), (\ref{Rel6}), (\ref{Rel7}). More precisely, let $ v^{1}_{1},\ldots, v^{1}_{m_{1}}$,$\ldots$,$v^{n}_{1},\ldots, v^{n}_{m_{n}} $ be elements of $ \mathcal{V} $, let $ \mathtt{s} $ be an element of $ \mathcal{RT}(n) $ and for any $ i\in\lbrace 1,\ldots,n\rbrace $ let $ \sTree_{i} $ be an element of $ \mathcal{RT}(m_{i}) $ such that 
\[ \Dec(\sTree_{i}(v^{i}_{1},\ldots, v^{i}_{m_{i}})) = F_{r_{i}}(v^{i}_{1};\mathfrak{p}^{i}_{1},\ldots, \mathfrak{p}^{i}_{r_{i}}), \] 
for some $ \mathfrak{p}^{i}_{j}\in \Gamma(\mathcal{RT},\mathcal{V}) $.

We consider an element of $\Gamma(\mathcal{RT}, \Gamma(\mathcal{RT}, \mathcal{V}))$
\[ T=\orb(\mathtt{s}(\orb(\sTree_{1}(v^{1}_{1},\ldots, v^{1}_{m_{1}})),\ldots,\orb(\sTree_{n}(v^{n}_{1},\ldots, v^{n}_{m_{n}})))). \] 
We want to compute the image of $ T $ under the map 
\[ \tilde{\mu}: \Gamma(PreLie,\Gamma(PreLie,V))\longrightarrow\Gamma(PreLie,V). \] 
Our strategy is to find a linear combination of simple elements with the same image of $ T $ under the map $ \tilde{\mu} $. We suppose $ m_{i}=max\lbrace m_{j}\vert j=1,\ldots,n  \rbrace $ i.e. the tree $ \sTree_{i} $ has the highest number of vertices. The normal form of $ T $ in $ \Gamma(\mathcal{RT},\Gamma(\mathcal{RT},\mathcal{V})) $ is a composition of corollas of the form:
\[ \orb(F_{s}(\orb \sTree_{i}(v^{i}_{1},\ldots, v^{i}_{m_{i}});\mathfrak{q}_{1},\ldots, \mathfrak{q}_{t})), \]
that is the image of
\[ S = \orb(F_{s}(\orb F_{r_{i}}(v^{i}_{1};\mathfrak{p}^{i}_{1},\ldots, \mathfrak{p}^{i}_{r_{i}});\mathfrak{q}_{1},\ldots, \mathfrak{q}_{t})) \]
under the map 
\[ id\circ\tilde{\mu}:\Gamma(PreLie,\Gamma(PreLie,\Gamma(PreLie,V)))\longrightarrow \Gamma(PreLie,\Gamma(PreLie,V)). \] 
Since $ \Gamma(PreLie,V) $ is a $ \Gamma PreLie $-algebra, the following diagram commutes:
\[\xymatrix{
\Gamma(PreLie,\Gamma(PreLie,\Gamma(PreLie,V)))\ar[d]_-{id\circ \tilde{\mu}}\ar[r]^-{\tilde{\mu}\circ id} & \Gamma(PreLie,\Gamma(PreLie,V))\ar[d]^-{\tilde{\mu}}\\
\Gamma(PreLie,\Gamma(PreLie,V)))\ar[r]_-{\tilde{\mu}} & \Gamma(PreLie,V).}\]

To compute the image of $ T $ we apply first $ \tilde{\mu}\circ id $ on $ S $ as composition of corollas. The result is a linear combination of elements of $ \Gamma(PreLie,\Gamma(PreLie,V)) $ whose normal forms are compositions of corollas which have as roots 
\[\orb(\sTree_{1}(v^{1}_{1},\ldots, v^{1}_{m_{1}})),\ldots,\Root (v^{i}_{1}),\mathfrak{p}^{i}_{1},\ldots, \mathfrak{p}^{i}_{r_{i}},\ldots \orb(\sTree_{n}(v^{n}_{1},\ldots, v^{n}_{m_{n}})).\] 
Since the number of vertices of  $ \mathfrak{p}^{i}_{j} $ is strictly smaller than $ m_{i} $, repeating the same computation inductively we obtain, in a finite number of passages, a sum of simple elements of $ \Gamma(PreLie,\Gamma(PreLie,V)) $. This procedure of computing $ T $ use just the compositions of corollas and it is performed the same way by using $ Cor $-algebra relations for corollas.
\end{proof}

This verification completes the proof of the following statement. 
\begin{thm}\label{ThmRep}
The construction of Proposition \ref{prp:PLCor} induces an isomorphism between the subcategories of $ \Gamma(PreLie,-) $-algebras and of $ Cor $-algebras formed by objects with a free $ \mathbb{K} $-modules structure.\hfill $\square$
\end{thm}
\begin{rmk}
The previous discussion shows that the functor $ \Gamma(PreLie, -) $ corresponds to an analyseur de Lazard (see \cite{article:Lazard55}) with non-commutative variables.
\end{rmk}
\section{Examples}\label{Applications}
In this last section we give some particular examples of $ \Gamma PreLie $-algebras.
\subsection{ Brace algebras are $ \Gamma PreLie $-algebras}
We recall the definition of the operad $ Brace $ and prove that any $ Brace $-algebra is a $ \Gamma PreLie $-algebra.
\begin{dfn}
Let $ V $ be a $\mathbb{K}$-module. It is a brace algebra if it is endowed with a sequence of operations $ \langle -;\underbrace{-,\ldots,-}_{n-1}\rangle:V^{\otimes n}\longrightarrow V $, subject to the following relations:
\begin{enumerate}
\item $ \langle x;\rangle=x $,
\item 
\begin{equation*}
\langle\langle x;y_{1},\ldots,y_{n}\rangle;z_{1}\ldots,z_{r}\rangle = \sum \langle x; Z_{1},\langle y_{1};Z_{2}\rangle ,Z_{3},\ldots, Z_{2n-1},\langle y_{n};Z_{2n}\rangle,Z_{2n+1} \rangle,
\end{equation*}
where the sum runs over the partitions of the ordered set $ \lbrace z_{1}\ldots,z_{r}\rbrace $ into (possibly empty) consecutive ordered intervals $ Z_{1}\sqcup\ldots\sqcup Z_{2n+1} $. 
\end{enumerate}

The operad corresponding to brace-algebras is denoted by $ Brace $.
\end{dfn}

The $ Brace $ algebras naturally appears in the study of Hochschild complex (see for example \cite{article:LadaMarkl05}).

We embed the operad $ PreLie $ into the operad $ Brace $:
\begin{equation*}
\psi:PreLie\hookrightarrow Brace
\end{equation*}
by sending $ \lbrace-,-\rbrace $ into $ \langle -,- \rangle $.
This inclusion induces a monomorphism from the $ \Gamma PreLie $ free algebra into the $ \Gamma Brace $ free algebra which is isomorphic to the $ Brace $ free algebra, since the symmetric action on the operad $ Brace $ is free. We accordingly have an inclusion from the $ \Gamma PreLie $ free algebra into the $ Brace $ free algebra. For more details see \cite{article:Chapoton02}.

This monomorphism is given by the following correspondence:
\begin{equation*}
\Flower{n}(x,\underbrace{y_{1}, \ldots, y_{1}}_{i_{1}},\ldots,\underbrace{y_{r}, \ldots, y_{r}}_{i_{r}}) \eqdef \sum_{\sigma\in \shuffle (i_{1},\ldots,i_{r})} \sigma* \langle x;\underbrace{y_{1}, \ldots, y_{1}}_{i_{1}},\ldots,\underbrace{y_{r}, \ldots, y_{r}}_{i_{r}}\rangle,
\end{equation*}
where $\shuffle (i_{1},\ldots,i_{r})$ is the set of the $(i_{1},\ldots,i_{r})$-shuffles.
More precisely:
\begin{dfn}
We call $ n $-planar-tree an $n$-tree with an order on the set $In(\Tree,i)$ for any vertex $ i $ of the $ n $-tree $ \Tree $. Let $ \lbrace PRT(n)\rbrace $ be the $ \SymGrp{} $-module with $ PRT(n) $ generated by the $ n $-planar labelled rooted trees. We define partial compositions
\begin{equation*}
 - \circ_{i} -:PRT(m)\otimes PRT(n)\longrightarrow PRT(n+m-1),
\end{equation*}
with $ 1 \leq i\leq m $ as follows:
\begin{equation*}
 (\Tree,ord(\Tree)) \circ_{i} (\upsilon,ord(\upsilon)) \eqdef \sum_{f:In(\Tree,i)\longrightarrow ( 1,\ldots, n )}\sum_{j_{1},\ldots, j_{n}}^{\vert s(1)+1\vert,\ldots,\vert s(n)+1\vert} (\Tree\circ_{i}^{f} \upsilon,ord(j_{1},\ldots, j_{n})) ,
\end{equation*}
where $ \Tree\circ_{i}^{f} \upsilon $ is the $ n+m-1 $-tree obtained by substituting the tree $ \upsilon $ to the $i$th vertex of the tree $\tau$, by attaching the outgoing edges of this vertex in $\tau$ to the root of $\upsilon$, and the ingoing edges accordantly with the attaching map $f$. The sum runs over all these attachment maps $f:In(\tau, i)\longrightarrow (1,\ldots n )$ preserving $ ord(\Tree) $ and $ ord(\upsilon) $.
\end{dfn}
\begin{lmm}
The operad $ Brace $ is isomorphic to the operad $ PRT $.
\end{lmm}
\begin{proof}
See \cite{article:Foissy10}.
\end{proof}
\begin{prp}\label{prp:braceGammaPreLie}
The action of symmetric groups on the operad $Brace$ is free. The brace algebras therefore coincide with $ \Gamma Brace $-algebras for any field and any brace-algebra inherits a $\Gamma PreLie$-algebra structure. More precisely we have a morphism from $ \Gamma (PreLie,V) $ into $ S(Brace,V) $, and we can make it explicit:
\begin{equation*}
\lbrace x;y_{1},\ldots ,y_{n}\rbrace_{r_{1},\ldots ,r_{n}}\mapsto\sum_{\sigma\in \shuffle(r_{1},\ldots ,r_{n})}\langle x;\overline{y_{\sigma(1)}},\ldots,\overline{y_{\sigma(r_{1}+\ldots +r_{n})}}\rangle,
\end{equation*}
where the ordered set $ (\overline{y_{1}},\ldots,\overline{y_{r_{1}+\ldots +r_{n}}}) $ is $ (\underbrace{y_{1},\ldots,y_{1}}_{r_{1}},\ldots,\underbrace{y_{n},\ldots,y_{n}}_{r_{n}}) $. \hfill $\square$
\end{prp}
\begin{ntt}
Let $ (P,\mu,\eta) $ be a connected operad. Let $\phi:\lbrace 1,\ldots r\rbrace \rightarrow \lbrace 1,\ldots, n\rbrace$ be an injective function. Let  $p\in P(n)$ and $q_{i}\in P(m_{i})$ where $i\in\lbrace 1,\ldots, r\rbrace$. We denote by $p\circ_{\phi}(q_{1},\ldots,q_{r})$ the following element of $P$:
\[\mu(p\otimes x_{1}\otimes\ldots\otimes x_{r})\]
where 
\[x_{j}=\begin{cases} q_{i}  & \hbox{if $ j=\phi(i)$ for some $i$,}  \\
1 & \hbox{the operadic unit, otherwise.}
\end{cases}\]
\end{ntt}
\begin{xmp}
Let $ (P,\mu,\eta) $ be a connected operad. It is a well known fact that the $\mathbb{K}$-module $\bigoplus_{i}P(i)$ is a $ PreLie $-algebra. This structure is induced by a $Brace$-algebra structure. Therefore the $PreLie$-algebra structure extends to a $\Gamma PreLie$-algebra structure. More explicitly:
\[\lbrace p;q_{1},\ldots,q_{m}\rbrace_{r_{1},\ldots r_{m}}=\sum_{\phi\in Sh_{n}(r_{1},\ldots, r_{m})}p\circ_{\phi}(\underbrace{q_{1},\ldots,q_{1}}_{r_{1}},\ldots,\underbrace{q_{m},\ldots,q_{m}}_{r_{m}})\]
where $Sh_{n}(r_{1},\ldots, r_{m})$ is the set of injective functions from $\lbrace 1,\ldots, r_{1}+\ldots +r_{m}\rbrace$ to $ \lbrace 1,\ldots, n\rbrace $ such that they are $(r_{1},\ldots ,r_{m})$-shuffle when we identify their image with $\lbrace 1,\ldots, r_{1}+\ldots +r_{m}\rbrace$.
\end{xmp}
Significant examples of $ PreLie $-algebras are associated to $ PreLie $-systems (see \cite{article:Gerstenhaber63} and \cite{GerstenhaberVoronov95}). We revisit the definition of this notion and we check that any $PreLie$-system gives rise to a $\Gamma PreLie$-algebra.
\begin{dfn}
Let $ \mathfrak{S}^{\centerdot} $ be a $ \mathbb{N} $-graded free $\mathbb{K}$-module. A $ PreLie $-system on $ \mathfrak{S}^{\centerdot} $ is a family of maps:
\begin{equation*}
\circ_{k}:\mathfrak{S}^{n}\otimes \mathfrak{S}^{m}\longrightarrow \mathfrak{S}^{n+m-1}
\end{equation*}
for any $ 1\leq k \leq n $, such that for any $ f\in\mathfrak{S}^{n}  $, $ g\in\mathfrak{S}^{m} $, and $ h\in\mathfrak{S}^{l} $ we have:
\begin{equation*}
f\circ_{u}(g\circ_{v}h=(f\circ_{u}g)\circ_{v+u-1}h
\end{equation*}
for any $ 1\leq u \leq n $ and $ 1\leq v \leq m $, and
\begin{equation*}
(f\circ_{u}g)\circ_{v+m-1}h=(f\circ_{v}h)\circ_{u}g
\end{equation*}
for any $ 1\leq u < v \leq n $.
\end{dfn}
\begin{prp}
Let $ f $ be an element of $\mathfrak{S}^{m} $ and $ g_{1},\ldots, g_{n} $ be elements of $ \mathfrak{S}^{\centerdot} $ with $ n\leq m $. We define:
\begin{equation*}
\langle f;g_{1},\ldots, g_{n} \rangle = \sum_{1\leq i_{1}<\ldots < i_{n}\leq m}(\ldots((f\circ_{i_{n}}g_{n})\ldots) \circ_{i_{1}}g_{1});
\end{equation*}
Then  $ \mathfrak{S}^{\centerdot} $ endowed with these operations is a $ Brace $-algebra, and hence inherits a $\Gamma PreLie$-algebra structure.\hfill $\square$
\end{prp}
\begin{xmp}
Let $ P $ a connected operad. The $PreLie$-algebra structure on the module $\bigoplus_{n}P(n)$ of the Example \ref{XmpPreLieAlg} (1) is clearly induced by a $PreLie$-system therefore it extends to a $\Gamma PreLie$-algebra structure.
\end{xmp}

\subsection{Dendriform algebras are $ \Gamma PreLie$-algebras}
I. Dokas proved in \cite{article:Dokas13} that dendriform algebras in positive characteristic admits a $p$-restricted $PreLie$-algebra structure (and hence a $\Lambda PreLie$-algebra structure by Theorem \ref{thmpRes}). We prove that any dendriform algebra is a $ \Gamma PreLie $-algebra.
\begin{dfn}
A dendriform algebra, denoted by $ Dend $-algebra, is a free $\mathbb{K}$-module $ A $ endowed with two binary operations $ < , >:A\otimes A\longrightarrow A $, such that:
\begin{gather*}
(x < y)< z = x <(y\ast z),\\
(x > y)< z = x > (y < z),\\
(x \ast y)> z = x >(y > z),\\
\end{gather*}
where $ x \ast y = x > y + y < x   $. It is easy to show that $ \ast $ is associative.

The category of dendriform algebras is governed by an operad denoted $ Dend $.
\end{dfn}

Dendriform algebras were introduced by J.L. Loday in \cite{incollection:Loday01} as Koszul dual of diassociative algebras in the study of K-Theory periodicity. They appear naturally in other fields such as combinatorial algebra, physics and algebraic topology. 
\begin{dfn}
Let $ (A,<,>) $ be a $ Dend $-algebra. We define the following binary operation $ \lbrace x,y\rbrace = x > y - y < x $. 
\end{dfn}
\begin{prp}
Let $ (A,<,>) $ be a $ Dend $-algebra. Then $ (A,\lbrace-,-\rbrace) $ is a $p$-restricted $PreLie$-algebra.
\end{prp}
\begin{proof}
See \cite{article:Dokas13}.
\end{proof}
We deduce from Theorem \ref{thmpRes} and the previous proposition that a $Dend$-algebra is a $\Lambda PreLie$-algebra.
\begin{prp}
Let $ V $ be a free $\mathbb{K}$-module then the $ PreLie $-algebra structure defined in $ S(Dend,V) $ extends to a $ \Gamma PreLie $-algebra structure.
\end{prp}
\begin{proof}
Let $ V $ be a free $\mathbb{K}$-module then the $ \Gamma PreLie $-algebra structure defined on $ S(Dend,V) $ is given by the inclusions $ PreLie \longrightarrow Brace \longrightarrow Dend$, and the construction of Proposition \ref{prp:braceGammaPreLie}.
\end{proof}

By the same kind of argument we prove that any Zinbiel algebra is a $ \Gamma(PreLie,-) $-algebra. Zinbiel algebras are encoded by the operad $ Zinb $ which was introduced by J.L. Loday in \cite{article:Loday95}, it is the Koszul dual of the operad $ Leib $ which encodes the Leibniz algebras. Then the cohomology of a $ Leib $-algebra inherits a $ Zinb $-algebra structure.
\begin{dfn}
Let $ A $ be a free $\mathbb{K}$-module, then it is a Zinbiel algebra if it is endowed with a bilinear product $ \circ $ such that:
\begin{equation*}
(a\circ b)\circ c = a\circ (b\circ c + c\circ b).
\end{equation*}
\end{dfn}
\begin{prp}
Let $ V $ be a free $\mathbb{K}$-module then $ S(Zinb,V) $ is a $ \Gamma PreLie $-algebra.
\end{prp}
\begin{proof}
This proposition follows from the inclusion of $ S(Dend,V) $ into $ S(Zinb,V) $.
\end{proof}
\paragraph*{Acknowledgements}
I would like to thank Benoit Fresse and Christine Vespa for their support throughout the writing of this paper.
\bibliographystyle{plain}
\bibliography{biblio}

\begin{thebibliography}{10}

\bibitem{article:Brouder00}
Ch. Brouder.
\newblock Runge--kutta methods and renormalization.
\newblock {\em The European Physical Journal C - Particles and Fields},
  12(3):521--534, 2000.

\bibitem{article:Chapoton02}
Fr{\'e}d{\'e}ric Chapoton.
\newblock Un th\'eor\`eme de {C}artier-{M}ilnor-{M}oore-{Q}uillen pour les
  big\`ebres dendriformes et les alg\`ebres braces.
\newblock {\em J. Pure Appl. Algebra}, 168(1):1--18, 2002.

\bibitem{article:ChapotonLivernet01}
Fr{\'e}d{\'e}ric Chapoton and Muriel Livernet.
\newblock Pre-{L}ie algebras and the rooted trees operad.
\newblock {\em Internat. Math. Res. Notices}, (8):395--408, 2001.

\bibitem{article:ConnesKreimer98}
Alain Connes and Dirk Kreimer.
\newblock Hopf algebras, renormalization and noncommutative geometry.
\newblock {\em Communications in Mathematical Physics}, 199(1):203--242, 1998.

\bibitem{article:Dokas13}
I.~Dokas.
\newblock Pre-{L}ie algebras in positive characteristic.
\newblock {\em J. Lie Theory}, 23(4):937--952, 2013.

\bibitem{article:Dzhumadil01}
Askar Dzhumadil'daev.
\newblock Jacobson formula for right-symmetric algebras in characteristic
  {$p$}.
\newblock {\em Comm. Algebra}, 29(9):3759--3771, 2001.
\newblock Special issue dedicated to Alexei Ivanovich Kostrikin.

\bibitem{article:Foissy10}
Lo{\"{\i}}c Foissy.
\newblock Free brace algebras are free pre-{L}ie algebras.
\newblock {\em Comm. Algebra}, 38(9):3358--3369, 2010.

\bibitem{article:Fresse00}
Benoit Fresse.
\newblock On the homotopy of simplicial algebras over an operad.
\newblock {\em Trans. Amer. Math. Soc.}, 352(9):4113--4141, 2000.

\bibitem{incollection:Fresse04}
Benoit Fresse.
\newblock Koszul duality of operads and homology of partition posets.
\newblock In {\em Homotopy theory: relations with algebraic geometry, group
  cohomology, and algebraic {$K$}-theory}, volume 346 of {\em Contemp. Math.},
  pages 115--215. Amer. Math. Soc., Providence, RI, 2004.

\bibitem{book:Fresse09}
Benoit Fresse.
\newblock {\em Modules over operads and functors}, volume 1967 of {\em Lecture
  Notes in Mathematics}.
\newblock Springer-Verlag, Berlin, 2009.

\bibitem{article:Gerstenhaber63}
Murray Gerstenhaber.
\newblock The cohomology structure of an associative ring.
\newblock {\em Ann. of Math. (2)}, 78:267--288, 1963.

\bibitem{GerstenhaberVoronov95}
Murray Gerstenhaber and Alexander~A. Voronov.
\newblock Homotopy {G}-algebras and moduli space operad.
\newblock {\em Int Math Res Notices}, 1995:141--153, 1995.

\bibitem{article:Hoffman03}
Michael~E. Hoffman.
\newblock Combinatorics of rooted trees and {H}opf algebras.
\newblock {\em Trans. Amer. Math. Soc.}, 355(9):3795--3811 (electronic), 2003.

\bibitem{book:Jacobson79}
Nathan Jacobson.
\newblock {\em Lie algebras}.
\newblock Dover Publications, Inc., New York, 1979.
\newblock Republication of the 1962 original.

\bibitem{article:LadaMarkl05}
Tom Lada and Martin Markl.
\newblock Symmetric brace algebras.
\newblock {\em Appl. Categ. Structures}, 13(4):351--370, 2005.

\bibitem{article:Lazard55}
Michel Lazard.
\newblock Lois de groupes et analyseurs.
\newblock {\em Ann. Sci. Ecole Norm. Sup. (3)}, 72:299--400, 1955.

\bibitem{article:Livernet06}
Muriel Livernet.
\newblock A rigidity theorem for pre-{L}ie algebras.
\newblock {\em J. Pure Appl. Algebra}, 207(1):1--18, 2006.

\bibitem{article:Loday95}
Jean-Louis Loday.
\newblock Cup-product for {L}eibniz cohomology and dual {L}eibniz algebras.
\newblock {\em Math. Scand.}, 77(2):189--196, 1995.

\bibitem{incollection:Loday01}
Jean-Louis Loday.
\newblock Dialgebras.
\newblock In {\em Dialgebras and related operads}, volume 1763 of {\em Lecture
  Notes in Math.}, pages 7--66. Springer, Berlin, 2001.

\bibitem{book:LodayVallette12}
Jean-Louis Loday and Bruno Vallette.
\newblock {\em Algebraic operads}, volume 346 of {\em Grundlehren der
  Mathematischen Wissenschaften [Fundamental Principles of Mathematical
  Sciences]}.
\newblock Springer, Heidelberg, 2012.

\bibitem{incollection:Machon}
Dominique Manchon.
\newblock A short survey on pre-{L}ie algebras.
\newblock In {\em Noncommutative geometry and physics: renormalisation,
  motives, index theory}, ESI Lect. Math. Phys., pages 89--102. Eur. Math.
  Soc., Z\"urich, 2011.

\bibitem{book:MarklSchniderStasheff02}
Martin Markl, Steve Schnider, and Jim Stasheff.
\newblock {\em Operads in algebra, topology and physics}, volume~96 of {\em
  Mathematical Surveys and Monographs}.
\newblock American Mathematical Society, Providence, RI, 2002.

\bibitem{incollection:Rota95}
Gian-Carlo Rota.
\newblock Baxter operators, an introduction.
\newblock In {\em Gian-{C}arlo {R}ota on combinatorics}, Contemp.
  Mathematicians, pages 504--512. Birkh\"auser Boston, Boston, MA, 1995.

\end{thebibliography}
\end{document}